\documentclass[12pt]{amsart}

\usepackage{amssymb,latexsym,amsthm, amsmath, comment, fullpage, wrapfig}
\usepackage[all]{xy}
\usepackage{graphicx, comment,  tikz-cd, libertine}

\ifx\pdfoutput\undefined
\usepackage{hyperref}
\else
\usepackage[pdftex,colorlinks=true,linkcolor=blue,urlcolor=blue]{hyperref}
\fi

\theoremstyle{plain}
\newtheorem{theorem}{Theorem}[section]

\newtheorem{corollary}{Corollary}[section]
\newtheorem{proposition}{Proposition}[section]

\newtheorem*{thm}{Theorem}

\newtheorem{lemma}{Lemma}[section]
\theoremstyle{definition}
\newtheorem{definition}{Definition}[section]
\newtheorem{remark}{Remark}[section]

\newtheorem*{ack}{Acknowledgements}
\newtheorem*{ques}{Question}
\newtheorem*{SA}{Standing assumptions}

\newcommand{\vertiii}[1]{{\left\vert\kern-0.25ex\left\vert\kern-0.25ex\left\vert #1 
    \right\vert\kern-0.25ex\right\vert\kern-0.25ex\right\vert}}
   
\begin{document}

\title[Non-local Dirichlet forms, Gibbs measures, and a cohomological Dirichlet principle for Cantor sets]
      {Non-local Dirichlet forms, Gibbs measures, and a cohomological Dirichlet principle for Cantor sets}
      \author{Rodrigo Trevi\~no}
      \address{Department of Mathematics, The University of Maryland, College Park, USA}
      \email{rodrigo@trevino.cat}
      \date{\today}
      \begin{abstract}
        In this paper I study properties of the generators $\triangle_\gamma$ of non-local Dirichlet forms $\mathcal{E}^\mu_\gamma$ on ultrametric spaces which are the path space of simple stationary Bratteli diagrams. The measures used to define the Dirichlet forms are taken to be the Gibbs measures $\mu_\psi$ associated to H\"older continuous potentials $\psi$ for one-sided shifts. I also define a cohomology $H_{lc}(X_B)$ for $X_B$ which can be seen as dual to the homology of Bowen and Franks. Besides studying spectral properties of $\triangle_\gamma$, I show that for $\gamma$ large enough (with sharp bounds depending on the diagram and the measure theoretic entropy $h_{\mu_\psi}$ of $\mu_\psi$) there is a unique $\mathcal{E}^\mu_\gamma$-minimizing representative of any class $c\in H_{lc}(X_B)$.
      \end{abstract}
      \maketitle

\section{Introduction and statement of results}

The classical Dirichlet principle asserts that harmonic functions are characterized variationally: among all functions satisfying a prescribed set of constraints, the harmonic one is the unique minimizer of the Dirichlet energy. This point of view extends beyond the usual boundary value setting. For instance, in the classical Hodge theorem, each de Rham cohomology class on a smooth compact manifold contains a unique preferred representative -- called the harmonic representative -- selected by an elliptic variational problem. A more general question is: if one is given a space of functions together with a finite-dimensional family of linear constraints, does each resulting affine class contain a unique function of least energy, and can that minimizer be interpreted as the natural “harmonic” representative of the class?

The purpose of this paper is to formulate and prove such a variational principle for Cantor sets. More precisely, using a non-local Dirichlet form to provide the energy and a finite-dimensional family of topological functionals to define the relevant cohomology classes, I show that each class admits a unique energy-minimizing representative. This is the sense in which the main result may be viewed as a cohomological Dirichlet principle for Cantor sets. To carry this out, one first has to determine 1) what should play the role of a Laplacian, or more fundamentally a Dirichlet energy, and 2) what a reasonable cohomology space should be. Once these are in place, one can ask whether each cohomology class admits a unique representative minimizing the corresponding energy.

The first issue is addressed through the theory of non-local Dirichlet forms. There is a rich literature of Dirichlet forms on arbitrary metric measure spaces, and in particular there has been recent progress on Dirichlet forms on ultrametric spaces. The second issue is addressed by using Cantor sets which are defined as the path spaces of Bratteli diagrams. Bratteli diagrams were introduced to study a particular type of $C^*$-algebras called approximately finite dimensional (AF) algebras. These ultrametric spaces come with a lot of structure, including topological invariants, and it is through these invariants that one can define cohomology spaces for these Cantor sets. Bratteli diagrams and their invariants have a long history of applications in dynamical systems, and the cohomology developed here can be seen as dual to a homology developed by Bowen and Franks \cite{BowenFranks:homology}.

Here I show that in this setting cohomology classes have unique energy-minimizing representatives.

\subsection{Statement of results}
Let $B$ be a simple, stationary Bratteli diagram (definitions are in \S \ref{sec:diagrams}). Under these hypotheses its path space $X_B$ is an ultrametric space which comes with a natural ultrametric as well. The building blocks of the space $X_B$ are cylinder sets of the form $C_e$, where $e$ is a finite path on $B$.

For any Borel probability measure $\mu$ with full support, for $\gamma>0$, consider the non-local Dirichlet form
$$\mathcal{E}_\gamma^\mu(f,g):= \frac12 \int_{X^2_B} \frac{(f(x)-f(y))(g(x)-g(y))}{d(x,y)^\gamma}\, d\mu (x)\, d\mu (y)$$ 
defined on the subspace $W_{\mu,\gamma}$ of $L^2_\mu$ of functions such that $\mathcal{E}_\gamma^\mu(f,f)<\infty$. By the standard theory of Dirichlet forms \cite{FOT:DF}, there is a non-negative definite, unbounded self-adjoint operator $-\triangle_\gamma^\mu$ on $W_{\mu,\gamma} $ such that
$$\mathcal{E}^\mu_\gamma(f,g) = -(\triangle_\gamma^\mu f, g)$$
for all $g\in W^\mu_\gamma $ and $f\in \mathrm{Dom}(\triangle_\gamma^\mu)\subset W_{\mu,\gamma} $. This operator is known as the \textbf{generator} of the Dirichlet form $\mathcal{E}^\mu_\gamma$. Here it will be also referred to as the Laplacian on $X_B$ corresponding to the form $\mathcal{E}^\mu_\gamma$.

While there is a distinguished probability measure of full support in this setting\footnote{It is distinguished because it is the unique measure which is invariant under the tail-equivalence relation.}, the setting is ideal to study the properties of non-local forms which use \textbf{Gibbs measures}, which are measures that play a prominent role in thermodynamic formalism. Given a H\"older function, $\psi:X_B\rightarrow \mathbb{R}$ it is a fundamental result that there exists a unique Borel probability measure, which is both an equilibrium state and a Gibbs measure for the (one-sided) shift $\sigma:X_B\rightarrow X_B$ (these are defined in \S\ref{sec:thermo}). These measures have full support, and thus can be used in the Dirichlet form above. Attached to any Gibbs measure $\mu_\psi$ is its \textbf{measure theoretic entropy} $h_{\mu_\psi}$. The distinguished measure $\mu$ mentioned above turns out to be both a Gibbs measure (associated to the function $\psi = 0$) and the measure of maximal entropy, that is, the measure for which $h_\mu$ is strictly greater than other measure theoretic entropies. This quantity is called the \textbf{topological entropy} and is denoted by $h_{top}$.

Now onto topology: consider the space $C_{lc}(X_B)$ of locally constant functions on $X_B$, which serves as the space of functions with the highest regularity, that is, as an analogue of $C^\infty(M)$. Associated to a Bratteli diagram $B$ is the locally finite (LF) $*$‑algebra $LF(B)=\varinjlim \mathcal M_k$ (dense in the AF algebra $AF(B)$) whose trace space $\mathcal T(B)=\varprojlim \mathrm{Tr}(\mathcal M_k)$ plays a central role. In the cases considered here, the trace-space will be finite dimensional. Each trace $\tau\in\mathcal T(B)$ induces a signed finitely additive measure on the clopen algebra of $X_B$ which, via the canonical embedding $j_B:C_{lc}(X_B)\to LF(B)$ that sends a cylinder indicator to the corresponding diagonal matrix unit, defines a linear functional $\mathcal{D}_\tau(f):= \tau(j_B(f))$ on $C_{lc}(X_B)$ (this is explained in detail in \S \ref{subsubsec:LF}). Traces on this dense $*$-algebra are cyclic cocycles, and thus the cohomology used here is one derived from the cyclic cohomology of the relevant LF algebras. Cyclic cohomology is a generalization of de-Rham cohomology for smooth manifolds \cite[Chapter 3]{Connes:book}, and thus its role here as a topological invariant is an appropriate substitute for de-Rham cohomology for smooth manifolds. 

The locally-constant cohomology is defined as follows: define $f\sim g$ if and only if $\mathcal{D}_\tau(f-g) = 0$ for all $\tau\in\mathcal{T}(B)$, then one can call $H_{lc}(X_B) = C_{lc}(X_B)/\sim$ the (locally-constant) cohomology of $X_B$. The motivation behind this definition is that the space of distributions $\{\mathcal{D}_\tau\}$ can be thought of as the space of all closed currents. As such, if a function is in the kernel of all closed currents then it should be thought of as representing a cohomologically trivial class. This allows one to bypass the need to define a coboundary map in order to define a cohomology, although one can do it by going explicitly through cyclic cohomology.

Besides being tied to the cyclic cohomology of the LF algebra defined by $B$, this is the cohomology which is dual to a homology introduced by Bowen and Franks in \cite{BowenFranks:homology}. These traces or finitely additive measures have a long history of playing significant roles in dynamical systems: they play a central role in Bufetov's study of minimal systems \cite{bufetov:limitVershik} as well as my own work \cite{T:TTT}, to name a few.

To connect this cohomology to the spaces $W_{\mu,\gamma}$ and thus to $\triangle_\gamma^\mu$, one first needs to prove that the distributions $\mathcal{D}_\tau$ extend to spaces of functions with sufficient regularity. Here I prove that for a Gibbs measure $\mu_\psi$ associated to a H\"older function $\psi$ and $\gamma$ large enough (with sharp bounds), the functionals $\mathcal{D}_\tau$ do extend to distributions on $W_{\mu_\psi,\gamma}$. As such, the relation $\sim$ can also be defined, giving the cohomology spaces $H_\gamma^\psi(X_B) = W_{\mu_\psi,\gamma}/\sim$, which are isomorphic to $H_{lc}(X_B)$. I then show that for $\gamma$ large enough every cohomology class has a unique $\mathcal{E}_\gamma^\psi$-minimizing representative.

Non-local Laplacians and Dirichlet forms on ultrametric spaces have been studied from several perspectives. In the metric–measure literature there are general results on non‑local Dirichlet forms that inform and motivate the present work \cite{KK:wavelets, Bendikov:ultra, BGHH:HK}. From the side of noncommutative geometry, Bellissard–Pearson introduced a spectral–triple framework for Laplacians on ultrametric Cantor sets \cite{BP:ultra}, a viewpoint developed further for Cantor systems arising from Bratteli diagrams in \cite{JS:LaplacianTiling, HKL:DirGraph}. Relatedly, recent work develops a logarithmic Laplace–Beltrami operator on Ahlfors‑regular spaces with compact/trace‑class heat semigroups and strong regularity/compatibility properties \cite{GM:logLap}, and uses these ideas to construct spectral triples and explicit heat operators for Cuntz–Krieger algebras and topological Markov chains \cite{GGM:HeatCuntz}. A complementary direction connects fractional Laplacians to Monge–Kantorovič (quantum) transport via Schatten‑class commutators and Weyl laws, with applications to algebraic and hyperbolic dynamics \cite{GM:QuantumMetrics}. In this paper I work entirely within the Bratteli diagram setting but without invoking spectral triples: I construct the fractional Laplacian from the non‑local Dirichlet form associated to the refining partitions, and then carry out an explicit spectral analysis of $\triangle_\gamma^{\mu_\psi}$ on $X_B$, including the identification of $L^2$ eigenfunction bases. Unlike previous works involving Bratteli diagrams, the results here are done for \emph{any} Gibbs measure, not just the measure of maximal entropy. See Remark \ref{rem:results} for more context.

\begin{SA}
    For the results of the paper, it is assumed that $B$ is a stationary simple Bratteli diagram where the metric scales like $\lambda^{-k}$ for cylinders defined by paths of length $k$ starting at $V_0$, where $\lambda>1$ is the Perron-Frobenius eigenvalue associated to $B$. The Gibbs measures $\mu_\psi$ will all be with respect to H\"older functions $\psi$. An important quantity will be the relative entropy between a Gibbs measure and the measure of maximal entropy. This will be denoted by
    $$d_\psi = \frac{h_{\mu_\psi}}{h_{top}}\in[0,1].$$
\end{SA}

To state the summary of the spectral properties of $\triangle_\gamma^\psi$, denote by $P_k$ the finite set of paths of length $k$ starting at $V_0\subset B$ (including paths of length 0, which are identified with elements of $V_0$). For any path $e\in P_k$ 
denote by $Y_\psi(e)\subset C_{lc}(X_B)$ the set of functions supported on the set of one-edge extensions $e'\in P_{k+1}$ of $e$ with zero average with respect to $\mu_\psi$. Moreover, let $W_{\psi,\gamma}^0$ be the set of functions $f\in W_{\mu_\psi,\gamma}$ of zero average, i.e., such that $\mu_\psi(f) = 0$. The spectral properties of $\triangle_\gamma^\psi$ are summarized in the following theorem, and keep in mind that there is a natural choice of an ultrametric for all Cantor sets considered here. As such, all statements are made assuming these canonical choice of ultrametric.
\begin{theorem}
\label{thm:spectrum}
Let $-\triangle_\gamma^\psi$ be the generator of $\mathcal{E}_\gamma^{\mu_\psi}$ on $W_{\psi,\gamma}^0$, where $\psi:X_B\rightarrow \mathbb{R}$ is a H\"older continuous function and $\mu_{\psi}$ is its unique Gibbs state. Then for $\gamma>d_\psi$:
\begin{enumerate}
    \item $\triangle_\gamma^\psi$ has a spectral gap: the smallest eigenvalue is
    $$\lambda_1 = \frac{\mu_\psi(X_B)}{\mathrm{diam}(X_B)^\gamma} = 1$$
    and its multiplicity depends on the first level of the Bratteli diagram.
    \item For $e\in P_k$, the space $Y_\psi(e)$ consists of eigenfunctions of $\triangle_\gamma^\psi$ with eigenvalue
    $$\lambda_\psi(e) = \frac{\mu_\psi(C_e)}{\mathrm{diam}(C_e)^\gamma} + \int_{X_B\setminus C_e} \frac{d\mu_\psi (z)}{\mathrm{dist}(C_e,z)}.$$
    Thus, $\lambda_\psi(e)\rightarrow \infty$ as the length of the path $e$ increases.
    \item (Poincar\'e inequality) For $f\in W_{\mu_\psi,\gamma}$:
    $$ \mathcal{E}_\gamma^\psi(f,f)\geq \|f-\mu(f)\|_{L^2_{\mu_\psi}}.$$
    \item (Weyl law) Let 
    $$N_\gamma^\psi(\Lambda) = |\left\{ \lambda\in\mathbb{R}: \lambda\mbox{ is an eigenvalue of }\triangle_\gamma^\psi \mbox{ with }\lambda\leq \Lambda\right\}|$$
    where the eigenvalues are counted with multiplicity. Then for $\gamma>d_\psi$ there is a constant $C_{\psi, \gamma}$ such that for all $\Lambda>1$
        $$C^{-1}_{\psi, \gamma} \Lambda^{\frac{1}{\gamma - d_\psi}}\leq  N_\gamma^\psi(\Lambda)\leq C_{\psi,\gamma} \Lambda^{\frac{1}{\gamma - d_\psi}}.$$
    \item (Heat kernel) The integral kernel $p_t^{\psi,\gamma}(x,y)$ of the heat operator $T_t^{\psi,\gamma}:=e^{t\triangle_\gamma^\psi}$ for $t\in (0,1]$ satisfies the two-sided estimate
$$c_1 t^{-\frac{d_\psi}{\gamma-d_\psi}}\left(1+\frac{d(x,y)}{t^{1/(\gamma-d_\psi)}}\right)^{-\gamma}\leq p_t^{\psi,\gamma}(x,y) \leq    c_2 t^{-\frac{d_\psi}{\gamma-d_\psi}}\left(1+\frac{d(x,y)}{t^{1/(\gamma-d_\psi)}}\right)^{-\gamma}$$
for some $c_1,c_2>0$. Moreover, it is continuous jointly in $x,y,t$ and H\"older continus in $x,y$.
    \end{enumerate}
\end{theorem}
\begin{remark}
\label{rem:results}
The results of Theorem \ref{thm:spectrum} likely all follow from existing and more general results in the literature, e.g. \cite{KK:wavelets, BK:hier, Bendikov:ultra, BGHH:HK, BP:ultra, JS:LaplacianTiling, HKL:DirGraph, GM:logLap, GGM:HeatCuntz, GM:QuantumMetrics}, but they are included here for completeness. The motivation here is not to present new spectral results, but to connect the domain of the non-local Dirichlet forms to the Besov-like spaces $\mathcal{S}_r$ introduced and studied in \cite{T:transversal, T:CohEq}, as well as to the cohomology developed here through a Hodge theorem. Still, Theorem \ref{thm:spectrum} is presented the way it is so that it can be of use to others working in the context of Bratteli diagrams and symbolic dynamics.
\end{remark}

In \S \ref{subsec:examples}, I work out the spectrum for $\triangle_\gamma^0$ for three examples of Bratteli diagrams which are related using the measure of maximal entropy $\mu$. These diagrams are related in that they have isomorphic topological invariants as well as deep dynamical connections (for the experts, they carry conjugate odometers). Interestingly, the spectra of $\triangle_\gamma^0$ in these different examples are different.

Stationary Bratteli diagrams can be defined by a single matrix $A$ with non-negative entries. If a diagram is simple, then the matrix $A$, by the Perron-Frobenius theorem, has its largest eigenvalue $\lambda$ to be unique and greater than 1. If $\mathrm{dim}(H_{lc}(X_B))>1$, denote by $\lambda_-$ the smallest non-zero absolute value of the eigenvalues of $A$. For $\gamma>2(1+d_\psi - \frac{\log\lambda_-}{\log\lambda})$, the distributions $\mathcal{D}_\tau$ extend to $W_{\mu_\psi,\gamma}$ (Lemma \ref{lem:extend}), and so
$$\mathcal{B}^\psi_\gamma := \left\{ f\in W_{\mu_\psi,\gamma} : \mathcal{D}_\tau(f) = 0 \mbox{ for all }\tau \right\}$$
serves as the space of coboundaries in
$$H^\psi_\gamma(X_B):= W_{\mu_\psi,\gamma} /\mathcal{B}_\gamma^\psi .$$
In this case ($\gamma>2(1+d_\psi - \frac{\log\lambda_-}{\log\lambda})$), a function $f\in W_{\mu_\psi,\gamma} $ is \textbf{$\mathcal{E}_\gamma^{\mu_\psi}$-minimizing} in its class if 
$$\mathcal{E}_\gamma^\psi(f,f) = \inf_{g\in [f]\in H^\psi_\gamma}\mathcal{E}_\gamma^\psi(g,g),$$
which implies $\mathcal{E}_\gamma^\psi(f,b) = 0$ for every $b\in \mathcal{B}_\gamma^\psi$ (see Theorem \ref{thm:harmonic}). Note that constant functions are minimizing in any setting ($\gamma\in\mathbb{R}$), so the space of harmonic functions always contains the constant functions. Thus the nontrivial cases are those where $\mathrm{dim}(H_{lc}(X_B))>1$.
\begin{theorem}[Cohomological Dirichlet principle for Cantor sets]
\label{thm:hodge}
Let $X_B$ be the path space of a stationary simple Bratteli diagram, let $\mu_\psi$ be the unique Gibbs measure on $X_B$ associated to the H\"older function $\psi$, and suppose that $\mathrm{dim}(H_{lc}(X_B))>1$. Consider the domain $W_{\psi,\gamma}$ of the non–local Dirichlet form $\mathcal{E}_\gamma^{\mu_\psi}$. Then for any $\gamma>2(1+d_\psi - \frac{\log\lambda_-}{\log\lambda})$ every cohomology class $c\in H_{lc}(X_B)$ has a unique $\mathcal{E}_\gamma^{\mu_\psi}$-minimizing representative in $W_{\mu_\psi,\gamma}$.
\end{theorem}

Some comments are now in order. First, one may wonder how general this statement is. On one hand Dirichlet forms can be defined on very general metric measure spaces, and so the very particular type of ultrametric spaces considered here seem like a very particular setting. This is true, but this happens because it is the only setting where there are spaces which serve as topological invariants of Cantor sets. Thus the generality has to be reduced in order to allow the ultrametric set to have enough structure to associate to them something that can genuinely be called cohomology that is finite-dimensional. However, even if setting is restricted to the path space of simple stationary Bratteli diagrams, Gibbs measures provide a wealth of measures with which to obtain different non-local operators.

Second, if a diagram is not stationary, some of the results may still hold. While it may not make sense to talk about Gibbs states in a non-stationary setting, at least in the renormalizable case \cite{T:CohEq}, there remains a unique probability measure which is a generalization of the measure of maximal entropy and the non-local Dirichlet form can still be defined using this measure. All of the arguments here would carry to the renormalizable cases for that measure, while leaving as interesting open cases the situations where the diagram is not obviously renormalizable.

Finally, given that Bratteli diagrams already carry invariants that have played a central role in classification problems, one may wonder whether the spectrum of $\triangle_\gamma^0$ is any sort of invariant for the diagrams. Given the results in \S \ref{subsec:examples}, I would venture to guess that the spectrum identifies diagrams exactly. Thus I wonder the following.
\begin{ques}[Can you hear the shape of a Cantor set?]
    Let $B_1$ and $B_2$ be two stationary, simple Bratteli diagrams and denote by $\sigma_1^\gamma$ and $\sigma_2^\gamma$ their respective spectra for $\triangle_\gamma$ using the natural ultrametric used in this paper. If $\sigma_1^\gamma = \sigma_2^\gamma$ for all $\gamma>1$, does it follow that $B_1 = B_2$?
\end{ques}
The hypothesis on the diagrams being stationary is crucial; see \S \ref{subsubsec:simplest}.

This paper is organized as follows. Section \ref{sec:diagrams} introduces Bratteli diagrams, the Cantor sets defined by its path spaces, the corresponding spaces of functions of varying regularity, the LF algebras and traces defined on them, and finaly the cohomologies that will be used in the paper. Section \ref{sec:thermo} reviews the relevant definitions of thermodynamic formalism and derives a few useful properties of Gibbs measure. Section \ref{sec:Dirichlet} introduces the non-local Dirichlet form $\mathcal{E}_\gamma^\psi$ and compares embeddings of involving its domain with the spaces of smooth functions introduced in \S \ref{sec:diagrams}. Section \ref{sec:spectral} studies the spectral properties of $\triangle_\gamma^\psi$, proves the main spectral result, Theorem \ref{thm:spectrum}, and includes examples with the spectra fully described. Finally, in Section \ref{sec:Hodge} the Hodge representation theorem is proved.

\begin{ack}
    I want to thank Ian Putnam and Giovanni Forni for helpful conversations throughout the course of writing this paper. This work was supported by NSF Career grant DMS-2143133. I am also grateful to the anonymous referee who brought \cite{KK:wavelets} to my attention.
\end{ack}

\noindent   \textbf{Disclosure of potential conflicts of interest or Competing Interest:} There are no conflicts of interest or competing interests to report.

\section{Bratteli Diagrams and LF algebras}
\label{sec:diagrams}

\begin{wrapfigure}{L}{0.4\textwidth}
    \includegraphics[width=1\linewidth]{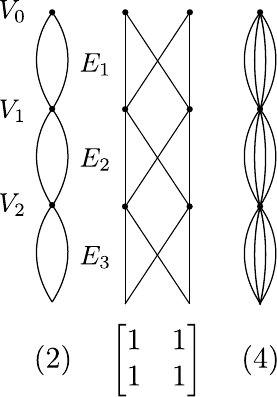}
    \caption{Three primitive simple diagrams which are related.}
    \label{fig:diagrams}
\end{wrapfigure}
In this section we recall and develop all the background material for Bratteli diagrams which will be used later.
\subsection{Basic definitions}
A \textbf{Bratteli diagram} is a indexed graph $B = (V,E)$ with 
$$\hspace{2.2in}V = \bigsqcup_{k\geq 0} V_k \hspace{.3in}\mbox{ and }\hspace{.3in} E = \bigsqcup_{k>0} E_k$$
for which there are maps $r,s:E\rightarrow V$, called the \textbf{range} and \textbf{source} maps which assign to each edge (for $s$) the vertex where the edge begins, or (for $r$) the vertex where the edge terminates. There will be a standing assumption that $r^{-1}(v)\neq \varnothing $ and $s^{-1}(v)\neq \varnothing $ for all $v\in V$. It will not be assumed here that $|V_0| = 1$.

For $k\in\mathbb{N}$, $k>2$, a \textbf{path of length $k$} is a collection of edges $(e_1,\dots ,e_k)\in E_{j}\times \cdots \times E_{j+k-1}$ such that $r(e_i) = s(e_{i+1})$ for all $0<i<k$. The set of paths of length $k$ with $s(e_1)\in V_0$ will be denoted by $P_k$. The set $P_1$, paths of length 1, can be identified with $E_1$ while paths of length 0, $P_0$, can be identified with $V_0$. Note that for any $\bar{e} = (e_1,\dots, e_k)\in P_k $ the maps $r,s$ can also be defined by $r(\bar{e}) = r(e_k)$ and $s(\bar{e}) = s(e_1)$. The set of infinite paths on $B$ is the collection of all paths $(e_1,e_2,\dots)\in \prod_{i>0} E_i$ such that $r(e_i) = s(e_{i+1})$ for all $i$. The set of all infinite paths will be denoted by $X_B$, and this is called the \textbf{path space of $B$}.

The set $X_B$ is topologized as follows. For any $k\in\mathbb{N}$ and $\bar{e}\in P_k$, define the cylinder set
$$C_{\bar{e}}:= \{(x_1,x_2,\dots)\in X_B: x_i = e_i\mbox{ for }1\leq i \leq k\}$$ 
and declare the topology of $X_B$ to be defined by declaring all cylinder sets $C_{\bar{e}}$ to be open, where $k\in \mathbb{N}$ and $\bar{e}\in P_k$. A finite intersection of cylinder sets as defined above are also cylinder sets. Since characteristic functions on cylinder sets will be heavily used, the following notational convention will be adopted: $\chi_{\bar{e}} := \chi_{C_{\bar{e}}}$ for $\bar{e}\in P_k$. Finally, for $k\in\mathbb{N}$ and $x\in X_B$, the \textbf{$k^{th}$ clopen set around $x$} is the cylinder set $C_k(x):= C_{(x_1,\dots, x_k)}$. If $e\in P_k$ and $e'\in s^{-1}(r(e))$, then the path $e$ can be extended by appending $e'$ to the end of $e$. Such concatenation of paths will be denoted by $e e'\in P_{k+1}$. 

Let $\mathcal{A}$ be the Borel $\sigma$-algebra defined by these clopen sets. A \textbf{tail-invariant measure} is a Borel measure $\mu$ such that $\mu(C_e) = \mu(C_{e'})$ whenever $r(e) = r(e')$ for any $e,e'\in P_k$. It is a standard fact which follows from Perron-Frobenius theory that if $B$ is a simple stationary Bratteli diagram then there is a unique tail-invariant probability measure $\mu$.

There is an alternate way to define $B$ and thus $X_B$. For $k\in \mathbb{N}$, let $A_k$ be a sequence of matrices, and $n_k$ a sequence of numbers such that, $A_k$ is $n_k\times n_{k-1}$. Then this information determines a Bratteli diagram with $|V_k| = n_k$ and the number $A_{ij}$ is the number of edges from $v_j\in V_{k-1}$ to $v_i\in V_k$. If there is a matrix $A$ such that $A_k = A$ for all $k$, then the Bratteli diagram it defines in this way is a \textbf{stationary Bratteli diagram}. The matrix $A$ is \textbf{primitive} if there is a $m\in\mathbb{N}$ such that $A^m$ has all positive entries. When a stationary Bratteli diagram is defined by a primitive matrix, it is said to be \textbf{simple}.

In this paper the only type of Bratteli diagram that will be considered is a stationary diagram defined by a primive matrix $A$. It is well known that under these conditions the resulting path space $X_B$ is a Cantor set and the basis of cylinder sets consists of clopen sets.

If $A$ is primitive matrix with non-negative integer entries, then the Perron-Frobenius theorem gives that there is a unique largest eigenvalue in modulus, which will be denoted by $\lambda>1$. The natural ultrametric on $X_B$ which will be used here is defined as\footnote{Note that in many references $2$ is used instead of $\lambda$ to define the metric. Here $\lambda$ is used because it is in a sense more natural, as it will make some computations easier later.}
$$d(x,y) = \lambda^{1-n(x,y)}$$
where $n(x,y)$ is the smallest index $i$s such that $x_i\neq y_i$. Note that for a cylinder set $C_k(x)$, this metric gives
\begin{equation}
    \label{eqn:diameter}
    \mathrm{diam}(C_k(x)) = \lambda^{-k}.
\end{equation}
Since $A$ is primitive and $X_B$ is a Cantor set, this metric is an ultrametric which generates the topology. The basic cylinder sets also have a measure that can be bounded using standard facts which follow from the Perron-Frobenius theorem. So in the case of stationary, simple Bratteli diagram there is a $K>1$ such that for any $k\geq 0$ and $e\in P_k$,
\begin{equation}
\label{eqn:parry}
K^{-1}\lambda^{-k}\leq \mu(C_e) \leq K\lambda^{-k}.
\end{equation}
Figure \ref{fig:diagrams} has three examples of simple, stationary Bratteli diagrams. The first one is the simplest one: the matrix which defines it is $(2)$. The other two are obtained from the first one through certain operations. The second one is obtained from the first by performing state splitting while the third one is obtained from the first one by telescoping. It is not important what these operations are because they will not be used. What is important is that these diagrams are related and thus share a good amount of information. Their spectral information will be compared in \S \ref{subsec:examples}.

\subsection{Function spaces}
\label{subsec:spaces}
For $k\in\mathbb{N}$, let $\mathcal{A}_k\subset \mathcal{A}$ be the $\sigma$-sub-algebra of $\mathcal{A}$ defined by the sets $\{C_{\bar{e}}\}_{\bar{e}\in P_k}$. The sequence $\{\mathcal{A}_k\}_k$ is increasing and thus, by the increasing martingale theorem, for any Borel probability measure $\mu$ and $f\in L^2_\mu$ there is a projection map $\Pi_k$ in $L^2_\mu$ given by the conditional expectation with respect to $\mathcal{A}_k$. Setting $\Pi_{0}  $ to be the trivial projection to zero in $L^2_\mu$, let $\delta_k:= \Pi_k- \Pi _{k-1}$ for $k\in \mathbb{N}$. Note that 
$$\Pi_k  = \sum_{i=1}^k \delta _i$$ 
and thus by the increasing martingale theorem every $f\in L^2_\mu$ is canonically decomposed as
\begin{equation}
    \label{eqn:decomp}
    f = \sum_{i>0} \delta_i f = \lim_{k\rightarrow \infty} \Pi _k f.
\end{equation}

The conditional expectation $\Pi _k$ can be explicitly written as
$$\Pi _k f (x) = \frac{1}{\mu(C_k(x))}\int_{C_k(x)} f \, d\mu .$$

A function $f$ is \textbf{locally-constant} if it is of the form of a finite sum $\sum \bar{a}_i \chi_{C_i}$ for cylinder sets $\{C_i\}$. This is equivalent to having a $k\in\mathbb{N}$ such that $f$ is of the form $f = \sum_j a_j  \chi_{\bar{e}_{j}}$ where $\bar{e}_{j}\in P_k$ and $a_j\in\mathbb{R}$. 

Set the image of the projection $\Pi_k$ by $L_k := \Pi_k L^2$. This space is spanned by locally constant functions given by paths in $P_k$ for $k\in\mathbb{N}$. In other words
$$L_k = \left\{f = \sum_{e\in P_k} \alpha(e)\chi_e\right\}.$$
For $k=0$, set $L_0$ to be the set of locally constant functions $f(x)$ which depend only on the source of $x$ in $V_0$. Now, for each $k\geq 0$ and $e\in P_k$, set
$$Y(e):= \left\{ u = \sum_{e'\in s^{-1}(r(e))}\alpha(e')\chi_{ee'}: \mu(u) = \sum_{e'\in s^{-1}(r(e))}\alpha(e')\mu(C_{ee'}) =  0  \right\}$$
to be the space of dimension $m(e) := |s^{-1}(r(e))| -1$. As such, denote
$$J_k := \bigoplus_{e\in P_k} Y(e)$$
and note that for $e\neq e'\in P_k$, the spaces $Y(e)$ and $Y(e')$ are orthogonal as they have disjoint supports.
\begin{lemma}
    For every $k\geq 0$, assuming the measure has full support, there is an orthogonal splitting $L_{k+1} = L_k\oplus J_k$. 
\end{lemma}
\begin{proof}
    Let $f\in L_k$ and $g\in J_k$. Then since $f$ is constant on $C_e$ for all $e\in P_k$:
    $$(f,g)_{L^2} = \Pi_kf \sum_{e'\in s^{-1}(r(e))}\alpha(e')\mu(C_{ee'}) = 0$$
    so $L_k\perp J_k$.
\end{proof}
\begin{corollary}
    There is an orthogonal decomposition in $L^2$ as
    $$J_k = L_0 \oplus \bigoplus_{n=0}^{k-1} J_n.$$
    In particular, there is an orthogonal decomposition of $L^2$ as
    $$L^2 = L_0 \oplus \bigoplus_{n\geq0} J_n.$$
\end{corollary}

Picking an orthonormal basis $\{\psi_{e,\ell}\}_{\ell=1}^{m(e)}$ for each $Y(e)$ and $e\in P_k$, the decomposition above gives that any $f\in L^2$ can be written as
$$f = \mu(f) + \sum_{k\geq 0 }\sum_{e\in P_k} \sum_{\ell=1}^{m(e)} \, (f,\psi_{e,\ell})_{L^2} \, \psi_{e,\ell}$$
with
\begin{equation}
\label{eqn:L2norm}
\|f\|^2_{L^2} = \mu(f)^2 + \sum_{k\geq 0 }\sum_{e\in P_k} \sum_{\ell=1}^{m(e)} \, |(f,\psi_{e,\ell})_{L^2}|^2 .
\end{equation}

The set of all locally-constant functions will be denoted by $C_{lc}(X_B)$. Note that by the the discussion above, any $f\in L^2_\mu$ is the limit of locally constant functions, since each $Y(e)$ consists of locally constant functions. Moreover, for any $f\in L^2_\mu$ and $k\in\mathbb{N}$, there exist $\bar{e}_1,\dots ,\bar{e}_l\in P_k$ and $a_1,\dots, a_l \in\mathbb{R}$ such that $\delta_k f = \sum_j a_j \chi_{\bar{e}_j}$. 

Now onto notions of regularity: for $r>0$, define the norm on $C_{lc}(X_B)$ to be
$$\|f\|_r := \sum_{k>0} \lambda^{rk} \|\delta_k f\|_\infty,$$
where $\lambda>1$ is the Perron-Frobenius eigenvalue of $A$, and define $\mathcal{S} _r(X_B)$ to be the Banach space obtained by the completion
$$\mathcal{S} _r(X_B):= \overline{C_{lc}(X_B)}^{\|\cdot \|_r }.$$
A version of these spaces were introduced and used in \cite{T:Transverse, T:CohEq}.
\begin{lemma}
\label{lem:cnsts}
For $u\in \mathcal{S}_r(X_B)$ and $\varepsilon\in (0,r)$ there is a constant $C_{u,\varepsilon}$ such that
$$\|u_k\|_\infty\leq C_{u,\varepsilon} \lambda^{-k(r-\varepsilon)}$$
for all $k\in\mathbb{N}$, where $u_k:= \delta_k u$.
\end{lemma}
\begin{proof}
    Let $I_\varepsilon(u)\subset \mathbb{N}$ be the set such that $\|u_k\|_\infty> \lambda^{-k(r-\varepsilon)}$. Then if the set $I_{\varepsilon}$ is infinite,
    $$\|u\|_r = \sum_{k>0}\lambda^{rk}\|u_k\|_\infty \geq \sum_{k\in I_{\varepsilon}(u)}\lambda^{rk}\|u_k\|_\infty \geq \sum_{k\in I_{\varepsilon}(u)}\lambda^{k\varepsilon} = \infty. $$
    So the set is finite and the result follows.
\end{proof}

It will be useful to connect the spaces $\mathcal{S} _r(X_B)$ with better-known function spaces -- H\"older functions. For $r>0$ denote by
$$|f|_r = \sup_{x\neq y}\frac{|f(x)-f(y)|}{d(x,y)^r}$$
the H\"older $r$-seminorm, and by $\|f\|_{H,r} = \|f\|_\infty + |f|_r$ the $r$-H\"older norm which defines the space of $r$-H\"older functions $H_r(X_B)$ with finite $\|f\|_r$ norm.
\begin{proposition}
    For $s>r$ there is a continuous inclusion $i_{s,r}:H_s(X_B)\rightarrow  \mathcal{S} _r(X_B)$ with 
    $$\|f\|_r  \leq \frac{2 \lambda^{r-s}}{1-\lambda^{r-s}}\|f\|_{H,s}.$$
\end{proposition}
\begin{proof}
    Let $f\in H_s(X_B)$. Then
    \begin{equation*}
    \begin{split}    
    \delta _i f(x) &= \Pi _if(x)-\Pi _{i-1}f(x) = \frac{1}{\mu(C_i(x))}\int_{C_i(x)} f(y) \, d\mu(y) - \frac{1}{\mu(C_{i-1}(x))}\int_{C_{i-1}(x)} f(y) \, d\mu(y) \\
    &= \frac{1}{\mu(C_i(x))}\int_{C_i(x)} f(y)-f(x) \, d\mu(y) - \frac{1}{\mu(C_{i-1}(x))}\int_{C_{i-1}(x)} f(y) -f(x)\, d\mu(y)
    \end{split} 
    \end{equation*}
    and so 
    \begin{equation*}
    \begin{split}    
    |\delta _i f(x)| &\leq  \frac{1}{\mu(C_i(x))}\int_{C_i(x)} |f(y)-f(x)| \, d\mu(y) + \frac{1}{\mu(C_{i-1}(x))}\int_{C_{i-1}(x)} |f(y) -f(x)|\, d\mu(y) \\
    &\leq  \frac{|f|_s}{\mu(C_i(x))}\int_{C_i(x)} d(x,y)^s \, d\mu(y) + \frac{|f|_s}{\mu(C_{i-1}(x))}\int_{C_{i-1}(x)} d(x,y)^s\, d\mu(y) \\
    &\leq |f|_s(\lambda^{-(i+1)s} + \lambda^{-is}) \leq 2 |f|_s \lambda^{-is},
    \end{split}  
    \end{equation*}
    from which it follows that $\|\delta_i f\|_\infty\leq 2|f|_s\lambda^{-is}$. Thus
    $$\|f\|_r  = \sum_{k>0}\lambda^{rk}\|\delta_k f\|_\infty \leq 2|f|_s\sum_{k>0} \lambda^{rk-sk} \leq \frac{2 \lambda^{r-s}}{1-\lambda^{r-s}}\|f\|_{H,s}<\infty$$
    as long as $s>r$.
\end{proof}
\subsection{LF algebras and traces}
Let $A$ be a primitive matrix which defines a stationary Bratteli diagram and let $N$ denote the size of the matrix. Here we will review the construction of the direct limit of matrix algebras defined by $A$; see \cite{davidson:book} for more details.

\subsubsection{$LF$-algebras}
\label{subsubsec:LF}
A \textbf{multimatrix algebra }is an algebra of the form $\mathcal{M} = M_{s_1}\oplus \cdots \oplus M_{s_m}$, where each $M_i$ is the finite dimensional algebra of $s_i\times s_i$ matrices over $\mathbb{C}$. Let $1_N := (1,\dots, 1)^T$ and set $h_k = A^k 1_N$. Note that by the Perron-Frobenius theorem, there exists a $C>0$ such that $h_k^i\leq C \lambda^k$ for all $k$ and $i$. For all $k\in\mathbb{N}$, let $\mathcal{M}_k $ be the multimatrix algebra consisting of matrix algebras of sized given by $h_k$, that is:
$$\mathcal{M}_k  := M_{h_k^1}\oplus \cdots  \oplus M_{h_k^N}.$$
Starting with $\mathcal{M}_0  = \mathbb{C}^N$, let $i_k:\mathcal{M}_{k-1} \rightarrow \mathcal{M}_k $ be the inclusion map defined by diagonally embedding $A_{ij}$ copies of $M_{h_{k-1}^j}$ into $M_{h_{k}^i}$. The order in which the embedding is done is not important as they are all unitarily equivalent. This family of inclusions defines a locally finite algebra through the direct limit
$$LF (A):= \lim_{\rightarrow} \left(\mathcal{M}_{k-1} ,i_k\right)$$
which is an infinite dimensional $*$-algebra.
\subsubsection{Traces}
\label{subsubsec:traces}
Every matrix algebra has a natural trace: if $M_n$ is the matrix algebra of $n\times n$ matrices over $\mathbb{C}$ then $\tau_n:M_n\rightarrow \mathbb{C}$ is the trace defined by adding up all the diagonal entries. This trace satisfies $\tau(AB) = \tau(BA)$, and up to a scalar it is the unique functional with that property. If the multimatrix $\mathcal{M}$ is made up of $N$ matrix algebras, then the trace space $\mathrm{Tr}(\mathcal{M})$ of $\mathcal{M}$ is isomorphic to $\mathbb{C}^n$, where each dimension comes from a trace in each of the matrix algebras. More precisely, for $(a_1,\dots, a_N)\in \mathcal{M}_k  = M_{h_k^1}\oplus \cdots \oplus M_{h_k^N}$, let $\tau_{k,j}$ be the trace on $M_{h_k^j}$:
$$\tau_{k,i}: a = (a_1,\dots, a_N) \mapsto \tau_{k,i}(a) = \sum_{j} (a_i)_{jj}.$$
The inclusion map $i_k:\mathcal{M}_{k-1} \rightarrow \mathcal{M}_k $ induces a dual map on traces,
$$i_k^*:\mathrm{Tr}(\mathcal{M}_{k} )\rightarrow \mathrm{Tr}(\mathcal{M}_{k-1} )$$
and so \textbf{the trace space} of $LF(A)$ is defined to be
\begin{equation}
    \label{eqn:invLim}
    \mathcal{T} _A:= \mathrm{Tr}(LF(A)) =  \lim_{\leftarrow} \left(\mathrm{Tr}(\mathcal{M}_k ),i_k^*\right).
\end{equation}
This space is endowed with the norm $\|\cdot\|$ defined, for $\tau = (\tau_0,\tau_1,\dots)\in\mathcal{T} _A$, as $\|\tau\| = |\tau_0|$, where $|\cdot|$ is the canonical norm in $\mathrm{Tr}(\mathcal{M}_0 ) = \mathbb{R}^N$. If $\tau = (\tau_0,\tau_1,\dots)\in \mathcal{T} _A$ then for any $k$, this will be written as
\begin{equation}
\label{eqn:trace}
\tau_k = \sum_{\ell = 1}^N b_{k,\ell}(\tau)\tau_{k,\ell}
\end{equation}
with $b_{k,\ell}(\tau)\in\mathbb{R}$. Let $b_k(\tau) = (b_{k,1},\dots, b_{k,N})^T$ be the set of vectors in the decomposition above.
\begin{lemma}
    \label{lem:bs}
    For $\tau\in \mathcal{T} _A$ and $b_k(\tau)$ as above, $b_{k-1}(\tau) = A^T b_{k}(\tau)$ for all $k\in\mathbb{N}$. Thus $\mathcal{T} _A$ can be identified with the inverse limit
    $$\mathrm{Tr}(LF(A))\cong \lim_{\leftarrow }\left(\mathbb{R}^N, A^T\right),$$
    and thus with the eventual range $(A^T)^N \mathbb{R}^N\cong (A^T)^N\mathrm{Tr}(\mathcal{M} _0) \subset \mathrm{Tr}(\mathcal{M} _0)$.
\end{lemma}
\begin{remark}
\label{rem:BF}
At this point one should compare the isomorphism in the Lemma to the definition of the "homology" groups of Bowen and Franks \cite[\S 2]{BowenFranks:homology}. As in the work of Bowen and Franks, in Theorem \ref{thm:extension} below these homology groups will be interpreted as finitely additive signed measures on locally constant functions.
\end{remark}
\begin{proof}
    Let $\tau = (\tau_0,\tau_1,\tau_2,\dots)\in \mathcal{T} _A$. The result will be shown for $k=1$; the general case follows by shifting indices. Let $a\in\mathcal{M}_0  = \mathbb{C}^N$. Expanding $\tau_0(a) = \tau_1(i_1(a))$ gives
    \begin{equation*}
        \begin{split}
            \tau_0(a) &=  \sum_{\ell} b_{0,\ell}(\tau) \tau_{0,\ell}(a) = i^*_1\tau_1(a) = \sum_{\ell} b_{1,\ell}(\tau) \tau_{1,\ell}(i_1(a)) = \sum_{\ell} b_{1,\ell}(\tau) \sum_{i}A_{\ell,i} \tau_{0,i}(a) \\
            &= \sum_{\ell} \sum_{i}b_{1,\ell}(\tau) A_{\ell,i} \tau_{0,i}(a)  = \sum_{i} \sum_{\ell}b_{1,\ell}(\tau) A_{\ell,i} \tau_{0,i}(a) = \sum_{i} \tau_{0,i}(a) \sum_{\ell}b_{1,\ell}(\tau) A_{\ell,i} \\
            &  =  \sum_{i} \tau_{0,i}(a) (A^Tb_1(\tau))_i.
        \end{split}
    \end{equation*}
Comparing the second and last entries shows that $b_{0,i}(\tau) = (A^Tb_1(\tau))_i$ for all $i$, from which the claims follow.
\end{proof}

There is a natural action of $A^T$ on $\mathrm{Tr}(LF(A))$, denoted by $A^T_*:\mathrm{Tr}(LF(A))\rightarrow \mathrm{Tr}(LF(A))$, and defined by
$$A^T_*:\tau = (\tau_0,\tau_1,\dots) \mapsto (A^T\tau_0, A^T\tau_1, A^T\tau_2,\dots).$$
The action $A^T_*$ on the finite-dimensional vector space $\mathrm{Tr}(LF(A))$ is nonsingular. As such, there is a decomposition of $\mathrm{Tr}(LF(A))$ into $A^T_*$-invariant subspaces as follows.
If $\lambda_1 = \lambda > |\lambda_2| > \cdots > |\lambda_m| = \lambda_- >0$ are the ordered nontrivial eigenvalues of $A$ in terms of absolute value, let
\begin{equation}
    \label{eqn:decomp}  
E_\ell(A) := \{ \tau\in \mathcal{T} _A: \tau_0 \mbox{ is a generalized eigenvector for }A^T_*\mbox{ with eigenvalue }\lambda_\ell \}
\end{equation}
$$d(A):= \mathrm{dim}(\mathcal{ T}^{ m}_A ) = \mathrm{dim}(\mathrm{Tr}(LF(A))).$$
\begin{theorem}
    \label{thm:extension}
    Each $\tau\in\mathcal{T} _A$ defines a signed, finitely-additive measure $\mu_\tau$ on $X_A $ which is also a distribution $\mathcal{D}_\tau:C_{lc}(X_B)\rightarrow \mathbb{R}$.  Moreover, if $d(A)>1$, $\mathcal{D}_\tau$ extends to a continuous functional $\mathcal{D}_\tau:\mathcal{S} _r(X_B)\rightarrow \mathbb{R}$ for all $r\geq 1-\frac{\log |\lambda_-|}{\log \lambda}$.
\end{theorem}
In particular, if $\lambda_- = |\lambda_m|$ is the smallest magnitude of any of the non-trivial eigenvalues of $A$, then $\mathcal{T} _A$ constitutes a maximal topologically-meaningful set of distributions for $\mathcal{S} _r(X_B)$ for $r\geq 1-\frac{\log |\lambda_-|}{\log \lambda}$. 
\begin{proof}
    First, for each vertex $v\in V$, chose an order on the finite set $r^{-1}(v)$ and note that this gives an order on the set $r^{-1}(\bar{e})$ for any $\bar{e}\in P_k$, with $k\in\mathbb{N}$. The proposition will be proved using this arbitrary choice of order, but it will be clear that the conclusion will not depend on the particular choice.
    
    The first step is to construct a map 
    $$j_A:C_{lc}(X_B)\rightarrow LF(A)$$
    such that $j_A(\chi_{\bar{e}})\in \mathcal{M}_k$ for $\bar{e}\in P_k$ and extend to all of $C_{lc}(X_B)$ by linearity. Suppose that $r(\bar{e})\in v_i\in V_k$. Then there are $h_k^i$ paths of length $k$ which end in $v_i$, and this set is ordered by the choice of order. Suppose $\bar{e}$ is in the $M^{th}$ position in this order. Then set $j_A(\chi_{\bar{e}})$ to be 1 in the $M^{th}$  diagonal entry of $M_{h_k^i}\subset \mathcal{M}_k $ and zero elsewhere. Thus, for $f = \sum_i a_i \chi_{\bar{e}_i}\in C_{lc}(X_B)$ we have
    $$j_A(f) := \sum_i a_i j_A(\chi_{\bar{e}_i})\in \mathcal{M}_k .$$
    As such, any $\tau\in\mathrm{Tr}(A)$ defines a functional $\mathcal{D}_\tau(f) = \tau\circ j_A(f)$ which is linear. Note that the choice of order is irrelevant, as the map $j_A$ places the relevant parts somewhere in the diagonal and $\tau$ is a linear combination of sums along diagonals on the different matrix algebras. More precisely, any two choices of orders would give images under $j_A$ which would be unitarily equivalent, so their images under a trace would be the same. Thus, the placement of $\chi_{\bar{e}}$ is not relevant. Also note that each $\tau\in\mathrm{Tr}(A)$ gives a finitely additive measure $\mu_\tau$ by $\mu_\tau(C_{\bar{e}}) = \tau\circ j_A(\chi_{\bar{e}})$.
 
    It remains to show that $\mathcal{D}_\tau$ is a bounded functional on $\mathcal{S} _r(X_B)$ for $r$ large enough whenever $d(A)>1$, which means that $0<\lambda_-<\lambda$. Let $\tau = (\tau_0,\tau_1,\tau_2,\dots)\in \mathcal{T}_A$. Then for each $k\in\mathbb{N}$, recalling (\ref{eqn:trace}),
    $$\tau_k = \sum_{\ell=1}^N b_{k,\ell}\tau_{k,\ell}$$
    for some $b_{k,\ell}\in\mathbb{R}$ and with $\tau_{k,\ell}$ being the canonical trace in $M_{h_k^\ell}$.  By Lemma \ref{lem:bs},
    $$\|\tau\| = |\tau_0| = |(A^T)^k b_k(\tau) |$$
    for all $k\in\mathbb{N}$. Since the Perron-Frobenius eigenvalue $\lambda>1$ is the largest of all of the eigenvalues for $A_*^T$ and $\lambda_- = |\lambda_m|$ the smallest (in norm), there is a $C_{\tau}$ such that $|b_k(\tau)|\leq C_{\tau} \lambda_-^{-k}$ for all $k>0$.
    
    If $f\in\mathcal{S} _r(X_B)$, then since $h_{k}^\ell$ records the sums along rows of the matrix $A^k$ and $\|\delta_k  f\|_\infty$ is the largest size (in absolute value) of the contribution of $f$ to $j_A(f)$ in $\mathcal{M}_k$, we have that
    \begin{equation*}
        |\tau_{k,\ell}\circ j_A(\delta_k f)| \leq h_{k}^{\ell} \|\delta_k  f\|_\infty.
    \end{equation*}
    Since $|h_k|\leq C_A \lambda^k$ for all $k$, putting all the bounds above together there is a $C_\tau '$ such that
    \begin{equation}
        \label{eqn:microBnd}
        |\tau_k\circ j_A(\delta _k f)| \leq \sum_{i=1}^N |b_{k,i}| | \tau_{k,i}\circ j_A(\delta_k  f)|\leq C_{\tau}' \lambda_-^{-k}\lambda^k \|\delta_k  f\|_{\infty} = C_{\tau}' \lambda^{rk} \|\delta_k  f\|_{\infty}
    \end{equation}
    for all $k>0$, where $r = 1-\frac{\log \lambda_-}{\log \lambda}$. Thus, for any $k$ we have
    $$\tau\circ j_A(\Pi _k f) = \sum_{i=1}^k \tau\circ j_A(\delta _i f) = \sum_{i=1}^k \tau_i\circ j_A(\delta _i f)$$
    which by (\ref{eqn:microBnd}) can be bounded as
    $$|\tau\circ j_A(\Pi _k f)| \leq  \sum_{i=1}^k |\tau_i\circ j_A(\delta _i f)|  \leq C_{\tau}''\sum_{i=1}^k \lambda^{rk} \|\delta_i  f\|_{\infty}  \leq C_{\tau}'' \|f\|_r.$$
    Thus, for such $r$, letting $k\rightarrow \infty$ gives that
    $$|\mathcal{D}_\tau(f)|\leq  C_{\tau}'' \|f\|_r $$
    which completes the proof.
\end{proof}
\subsection{Cohomology for $X_B$}
This section will introduce and develop some topological invariants for $X_B$ using the results from previous sections.  I should mention that this begins with a cohomology which is dual to the homology of Bowen and Franks \cite[\S 2]{BowenFranks:homology}. In fact, some of the ideas in \S \ref{subsubsec:traces} can already be traced back to their work. In this section the connection is made more explicit.

First, define
$$\mathcal{B}_A := \{f\in C_{lc}(X_B): \mathcal{D}_\tau(f) = 0 \mbox{ for all }\tau\in\mathcal{ T}_A\}.$$
and define the relation $\sim$ on $C_{lc}(X_B)$ by $f\sim g$ if and only if $f-g\in\mathcal{B}_A$. Equivalence classes will be denoted by $[\cdot]$.

It is worthwhile exploring these relations. For $p,q>0$, let $x = (x_p,\dots, x_q)$ and $y = (y_p,\dots, y_q)$ be two paths in $E_p\times \cdots \times E_q$, and $\chi_x$, $\chi_y$ the characteristic functions of cylinder sets $C_x$ and $C_y$ obtained by matching $x$ and $y$, respectively, from indices $p$ to $q$. Note that $\chi_x$ is a sum of characteristic functions of paths in $E_q$, namely
$$\chi_x = \sum_{\overset{e'\in E_{p-1},}{ r(e') = s(x)}}\chi_{e'x},$$
for any trace in $\tau \in \mathrm{Tr}(\mathcal{M}_q)$ written as
$$\tau = \sum_{i} b_{q,i}(\tau)\tau_{q,i},$$
it follows that
$$\tau(j_A(\chi_x)) = \sum_{\overset{e'\in E_{p-1},}{ r(e') = s(x)}} b_{q,r(x)}(\tau) \tau_{q,r(x)}(j_A(\chi_{e'x})) = b_{q,r(x)}(\tau) |\{\bar{e}\in E_p:r(\bar{e}) = s(x)\}|$$
and so $[\chi_{x}]= [\chi_{y}]$ if and only if $r(x) = r(y)$ and $s(x) = s(y)$.

\begin{definition}
The \textbf{(locally constant) cohomology space of $X_B$} is the vector space 
$$H^{0}_{lc}(X_B):= C_{lc}(X_B)/\sim.$$
\end{definition}
\begin{remark}
     Although $H^0_{lc}(X_B)$ is called the cohomology here it is not properly a cohomology theory in the sense that it satisfies something resembling the Eilenberg–Steenrod axioms. However, I take the liberty in calling it the cohomology because 1) it is dual to the Bowen-Franks homology (recall Remark \ref{rem:BF}) in that a function is considered to be exact if it is in the kernel of what are considered closed currents (the functionals $\mathcal{D}_\tau$) and 2) it is connected to cyclic cohomology.
\end{remark}
\begin{remark}
    The relations $\sim$ can also be defined on $C(X_B;\mathbb{Z})$ which yield analogous cohomology groups $H^{0}_{lc}(X_B;\mathbb{Z})$.
\end{remark}
    There is a homomorphism 
    $$q_A:C_{lc}(X_B)\rightarrow \mathcal{T}^{*}_A := \mathrm{Hom}(\mathcal{T} _A,\mathbb{R})$$
    defined as follows. For each $f\in C_{lc}(X_B)$, consider the functional on $\mathcal{T} _A$ defined by $\tau\mapsto \tau(j_A(f))$ for all $\tau\in \mathcal{T} _A$. This is a linear functional on the finite dimensional vector space $\mathcal{T} _A$, and thus can be identified with an element $q_A(f)\in \mathcal{T}^{*}_A$, so  $\langle \tau, q_A(f)\rangle = \mathcal{D}_\tau(f)$. The map $q_A$ is surjective and, as such, it implies that $[f]$ is trivial in $H^0_{lc}(X_B)$ if and only if $q_A(f) = 0 \in \mathcal{T}^{*}_A$.
\begin{theorem}
\label{thm:DirLim}
    There is an isomorphism 
    $$H_{lc}^0(X_B)\cong \lim_{\rightarrow} \left( \mathbb{R}^N,A\right).$$
    Moreover, if $d(A)>1$, for $r>1-\frac{\log|\lambda_m|}{\log\lambda}$, defining $H_r^0(X_B) := \mathcal{S}_r (X_B)/\sim$, there is an isomorphism of vector spaces
    $$H^0_{lc}(X_B)\cong H^0_r(X_B).$$
\end{theorem}
\begin{proof}
     First, by the construction and properties of $q_A$, $H^0_{lc}(X_B) = C_{lc}(X_B)/\ker (q_A) $ can be identified with $\mathcal{T}^{*}_A$ which by (\ref{eqn:invLim}) and Lemma \ref{lem:bs} is isomorphic to the direct limit of $\mathbb{R}^N$ through $A$.

    Now, if $d(A)>1$, let $f\in \mathcal{S}_r(X_B)$ for $r> 1-\frac{\lambda_-}{\lambda}$. If there is a $h_f\in C_{lc}(X_B)$ such that $\mathcal{D}_\tau(h_f) = \mathcal{D}_\tau(f)$ for all $\tau\in \mathcal{T}_A$, then $h_f - f\in \mathcal{B}_A$, meaning that $H^0_{lc}(X_B)\cong H^0_r(X_B)$.
    Finding such a $h_f$ is possible and is done as follows.  Let $\tau_1,\dots,\tau_{d(A)}$ be an orthonormal basis for $\mathcal{T}_A$. For $i = 1,\dots, N = |V_0|$, let 
    $$g_i = \sum_{e'\in s^{-1}(v_i)} \chi_{e'}\in C_{lc}(X_B)$$
    be a collection of functions which are "supported" on $V_0$, that is, $g_i(x)$ only depends only on $s(x)\in V_0$. There is a linear combination of the $\{g_i\}$ which forms an orthonormal basis $f_1,\dots, f_{d(A)}$ of $\mathcal{T}^*_A$ which by the previous paragraph can be identified with the eventual range $A^N\mathbb{R}^N\subset \mathbb{R}^N$ of $A$. This basis can be picked to be dual to the $\tau_i$, that is, so that $\mathcal{D}_{\tau_i}(f_j) = \delta_{ij}$. Thus, given $f\in \mathcal{S}_r(X_B)$ with $r> 1-\frac{\lambda_-}{\lambda}$, write
    $$h_f = \sum_{i=1}^{d(A)} \mathcal{D}_{\tau_i}(f) f_i.$$
    By construction, since $h_f$ is a linear combination of the $\{g_i\}\subset C_{lc}(X_B)$, $h_f\in C_{lc}(X_B)$. Also, by construction, $\mathcal{D}_\tau(h_f) = \mathcal{D}_\tau(f)$ for all $\tau\in\mathcal{T}_A$. So $h_f - f\in \mathcal{B}_A$ and $H^0_{lc}(X_B)\cong H^0_r(X_B)$.
\end{proof}
\section{Gibbs measures}
\label{sec:thermo}
This section goes over the necessary background in thermodynamic formalism to derive enough properties of Gibbs measures which will be used later. The standard reference is \cite{Bowen:book}.

Let $B$ be a simple, stationary Bratteli diagram and $X_B$ its path space. The \textbf{shift map} is the function $\sigma:X_B\rightarrow X_B$ defined as
$$\sigma:x = (x_1,x_2,x_3,\dots) \mapsto (x_2,x_3,x_4,\dots)$$
for any $x = (x_1,x_2,x_3,\dots)\in X_B$. Since $B$ is simple the matrix which defines $B$ is primitive and so this implies that $\sigma$ is topologically mixing. 

A Borel probability measure $\mu$ on $X_B$ is $\sigma$-invariant if $\sigma_*\mu = \mu$. Such a measure satisfies the \textbf{Gibbs property with potential $\psi$} if there exists constants $C_\psi>0$ and $P\in\mathbb{R}$ such that
 $$C_f^{-1} \leq \frac{\mu(C_e)}{\displaystyle \exp \sum_{i=0}^{n-1}(\psi \circ \sigma^i(x)-P)}\leq C_f$$
for all $x\in C_e $ with $e\in P_n$.  The unique tail-invariant probability measure is indeed a Gibbs measure. Indeed by (\ref{eqn:parry}) it follows that this measure is a Gibbs measure for potential zero and $P = \log \lambda$. In this setting, this is called the \textbf{Parry measure}. 

Every $\sigma$-invariant measure has a number $h_\mu$ associated to it, its \textbf{(measure theoretic) entropy}. To define it, first, let $\mathcal{P}$ be a finite partition of $X_B$ and set 
$$H_\mu(\mathcal{P}) = -\sum_{P\in\mathcal{P}}\mu(P)\log(\mu(P))$$
where the sum ignores elements with zero measure as $x\log x\rightarrow 0$ as $x\rightarrow 0^+$. With this defined, the \textbf{entropy of $\mu$ with respect to $\mathcal{P}$} is
$$H_\mu(\mathcal{P}) = \lim_{n\rightarrow \infty }\frac1n  H_\mu\left( \bigvee_{k=0}^{n-1} \sigma^{-k}\mathcal{P}\right).$$
Finally, the measure theoretic entropy of $\mu$ is defined to be
$$h_\mu = \sup \left\{H_\mu(\mathcal{P}):\mbox{$\mathcal{P}$ is a finite partition of $X_B$}  \right\}.$$
From this, the variational pressure can be defined as
$$P_\psi = \sup \left\{ h_\mu + \int_{X_B}\psi\, d\mu: \mu\mbox{ is a $\sigma$-invariant probability measure} \right\}.$$
Any measure $\mu$ which achieves the supremum, that is, such that $P_\psi = h_\mu+\int_{X_B}\psi\, d\mu$ is called an \textbf{equilibrium state}. The following is a seminal result in the area; its proof can be found in \cite{Bowen:book}.
\begin{thm}
    For every H\"older function $\psi:X_B\rightarrow \mathbb{R}$ there exists a unique Borel probability $\sigma$-invariant measure $\mu_\psi$ satisfying the Gibbs property which is also a unique equilibrium state. The Gibbs property is satisfied with the variational pressure $P_\psi$ of $\psi$.
\end{thm}

Any such measure is called a \textbf{Gibbs measure}.  Note that for the Parry measure $\mu$, which corresponds to potential $\psi = 0$, is the measure for which $h_\mu$ is largest, i.e., it is the \textbf{measure of maximal entropy}. The largest such entropy is called the \textbf{topological entropy} of $\sigma$ and satisfies
$$h_{top} = \sup \left\{ h_\mu: \mu \mbox{ is a $\sigma$-invariant probability measure} \right\}.$$
In this case relevant to this paper, the shift $\sigma:X_B\rightarrow X_B$, it is a basic fact that $h_{top} = \log \lambda$.
\subsection{Gibbs measures and cylinders}
\label{subsec:gibbsCyl}
Let $\psi:X_B\rightarrow \mathbb{R}$ be a H\"older function and $\mu_{\psi}$ be its Gibbs measure. By the Birkhoff ergodic theorem, for $\mu_{\psi}$-almost every $x\in X_B$ and $f\in L^2_{\mu_\psi}$
$$\frac1n \sum_{k=0}^{n-1} f \circ \sigma^k(x)\longrightarrow \int_{X_B}f\, d\mu_\psi.$$
Letting $f = \psi$, by the Gibbs property,
$$\displaystyle -\frac{\log(\mu_\psi(C_n(x))}{n} = P_\psi - \frac1n \sum_{k=0}^{n-1}\sigma^k\circ \psi (x) + O(n^{-1})$$
where the implied constant depends only on the constants from the Gibbs property. Recalling (\ref{eqn:diameter}), the diameters of $C_k(x)$ are exactly $\lambda^{-k}$, and so the \textbf{relative dimension} of $\mu_\psi$ is defined to be
$$d_\psi:=\lim_{n\rightarrow \infty}\frac{\log \mu_{\psi}(C_n(x))}{\log(\mathrm{diam}(C_n(x)))} = \frac{h_{\mu_\psi}}{\log \lambda} = \frac{h_{\mu_\psi}}{h_{top}}$$
which is constant $\mu_\psi$-almost everywhere. Note that when $\mu$ is the measure of maximal entropy, it has maximal relative dimension $d_\mu = 1$. Moreover, since $\psi$ is H\"older, it follows that there exist constants $0<r_1<r_2$ and a bounded measurable function $\varpi:X_B\rightarrow [r_1,r_2]$ such that 
\begin{equation}
\label{eqn:distort}
\mu_\psi(C_n(x)) = \varpi(x) \cdot \mathrm{diam}(C_n(x))^{d_\psi} = \varpi(x)\cdot \lambda^{-nd_\psi}.
\end{equation}
Finally, it follows that there exist constants $0<c_1<c_2$ such that for any $x\in X_B$,
\begin{equation}
\label{eqn:childBound}
c_1\leq \frac{\mu_\psi(C_{k+1}(x))}{\mu_\psi(C_k(x))}\leq c_2.
\end{equation}
Indeed, using the Gibbs property, it is immediate to verify that
$$C^{-2} e^{\psi^- - P_\psi }\leq \frac{\mu_\psi(C_{k+1}(x))}{\mu_\psi(C_k(x))}\leq C^{2} e^{\psi^+ - P_\psi }$$
where $\psi^- = \min \psi(x)$ and $\psi^+ = \max \psi(x)$.
\section{Non-local Dirichlet forms}
\label{sec:Dirichlet}
Let $B$ be a simple, stationary Bratteli diagram and $\mu$ a Borel probability measure on $X_B$ with full support. For $\gamma>0$, define the non-local Dirichlet form as
$$\mathcal{E}^\mu_\gamma(f,g) = \frac12 \int_{X^2_B} \frac{(f(x)-f(y))(g(x)-g(y))}{d(x,y)^\gamma}\, d\mu(x)\, d\mu(y)$$
for $f,g\in C_{lc}(X_B)$. Moreover, on $C_{lc}(X_B)$, introduce the norm $\|\cdot \|_{\mu,\gamma}$ defined as
$$\|f\|_{\mu, \gamma}:= \mathcal{E}^\mu_\gamma(f,f)^{\frac12}.$$
Now complete the space of locally constant functions with respect to this norm to obtain
$$W_{\mu,\gamma}:= \overline{C_{lc}(X_B)}^{\|\cdot \|_{\mu, \gamma}}.$$
This space inherits the structure of a Hilbert space when the inner product is defined as
$$\langle u,v\rangle _{\mu,\gamma} = (u,v)_{L_2(\mu)} + \mathcal{E}^\mu_\gamma(u,v).$$
The rest of the section is devoted to understanding how the spaces $\mathcal{S}_r$ and $W_{\mu,\gamma}$ compare for any Gibbs measure $\mu$. To help reduce annoying notation, set 
$$\mathcal{E}^\psi_\gamma := \mathcal{E}^{\mu_\psi}_\gamma \hspace{1in} \mbox{ and } \hspace{1in} W_{\psi,\gamma} := W_{\mu_\psi,\gamma} $$
for a H\"older function $\psi$ and its Gibbs measure $\mu_\psi$. Finally, 
set
$$W_{\mu, \gamma}^0:= \{f\in W_{\mu,\gamma}: \mu(f) = 0 \}.$$
\begin{proposition}
\label{prop:inclusions}
    Let $\mu_\psi$ be a Gibbs measure with H\"older potential $\psi$ and $d_\psi$ the relative dimension of $\mu_\psi$. The following embeddings hold:
    \begin{enumerate}
        \item $\mathcal{S}_r\subset W_{\psi,\gamma}$ whenever $r>\frac{\gamma-d_\psi}{2}$, and
        \item $W_{\psi,\gamma}\subset \mathcal{S}_r $ continuously whenever $\gamma>2(r+d_\psi)$.
    \end{enumerate}
\end{proposition}
Before proving the proposition, a lemma needs to be proved.
\begin{lemma}
\label{lem:L2W}
Let $\mu_\psi$ be the Gibbs measure for the H\"older potential $\psi$. For any $\gamma>0$ there is a constant $C_\psi$ such that for all $f\in W_{\psi,\gamma} $:
$$\sum_{k> 0 }\lambda^{(\gamma -d_\psi) k}\|\delta_k f\|_{L^2}^2 \leq C_\psi \mathcal{E}^\psi _\gamma(f,f).$$
\end{lemma}
\begin{proof}
Since the form $\mathcal{E}^\psi _\gamma$ is an integral over $X^2_B$, the full measure set $X^2_B\setminus \mathrm{diag}(X^2_B)$ will be written as a disjoint union as follows:
\begin{equation}
\label{eqn:decomp}
    X^2_B\setminus \mathrm{diag}(X^2_B) = \bigsqcup_{k\geq 0} \bigsqcup_{e\in P_k} \bigsqcup_{\overset{e_1,e_1\in s^{-1}(r(e))}{e_1\neq e_2}} C_{ee_1}\times C_{ee_2}.
\end{equation}
Indeed, for every $k>0$, consider the partition of $X_B$ into cylinders given by paths of length $k$. For each $k\in\mathbb{N}$, label all the clopen sets from this partition as $C_e$ depending on the path. At level $k$, we consider all $C_{e}\times C_{e'}$ as long as $e\neq e'\in P_k$. The case $e = e'$ is broken down to avoid the diagonal, and so for every $e\in P_k$ we consider all possible one-edge extensions from $s^{-1}(r(e))$, and from these we get the sets $C_{ee_1}\times C_{ee_2}$. Doing this for at every level, all clopen sets that do not intersect the diagonal are exhausted. As such, if the integral over $C_{ee_1}\times C_{ee_2}$ can be bound adequately, by (\ref{eqn:decomp}), they can all be added up to obtain the desired bound. Note that the decomposition starts at $k = 0$ and this corresponds to paths of zero length, meaning that their source and range are is the same vertex in $V_0$, which is taken to be a path of length zero.

To this end, first note that for $e\in P_k$ and $e_1\neq e_2 \in s^{-1}(r(e))$, Jensen's inequality gives
\begin{equation}
\label{eqn:Jensen}
\begin{split}
\int_{C_{ee_1}\times C_{ee_2}}\frac{(f(x)-f(y))^2}{\mathrm{dist}(x,y)^\gamma}\, d\mu(x) \, d\mu(y) &= \lambda^{\gamma k} \int_{C_{ee_1}\times C_{ee_2}} (f(x)-f(y))^2\, d\mu(x) \, d\mu(y) \\ 
&\geq   \frac{\lambda^{\gamma k}}{\mu(C_{ee_1})\mu(C_{ee_2})} \left( \int_{C_{ee_1}\times C_{ee_2}}f(x) - f(y)\,d\mu^2(x,y)  \right)^2 \\
& = \lambda^{\gamma k}\mu(C_{ee_1})\mu(C_{ee_2})\left(\Pi_{k+1}f(z_{ee_1})- \Pi_{k+1}f(z_{ee_2})\right)^2,
\end{split}
\end{equation}
where $z_{ee_i}\in C_{ee_i}$.  
A tedious but straight forward computation shows that for $e\in P_k$
\begin{equation}
    \label{eqn:app}
    \frac12 \sum_{\overset{e_1,e_1\in s^{-1}(r(e))}{e_1\neq e_2}}\mu_\psi(C_{ee_1})\mu_\psi(C_{ee_2})\left(\Pi_{k+1}f(z_{ee_1})- \Pi_{k+1}f(z_{ee_2})\right)^2  = \mu_\psi(C_e) \|\delta_{k+1} f \|^2_{L^2(C_e)},
\end{equation}
see Appendix \ref{app:tedious} for the derivation. Thus (\ref{eqn:distort}) and (\ref{eqn:decomp})-(\ref{eqn:app}) together give
\begin{align*}
    \mathcal{E}^\psi _\gamma(f,f) &\geq  \sum_{k\geq 0}  \lambda^{\gamma k}\sum_{e\in P_k}  \mu_\psi(C_e) \|\delta_{k+1} f\|^2_{L^2(C_e)} \geq C_\psi \sum_{k>  0}\lambda^{(\gamma-d_\psi) k}\|\delta_k f\|^2_{L^2},
\end{align*}
for a constant $C_\psi$ which bounds both the number of paths of length $k$ as well as the measure of the sets $C_e$ for $e\in P_k$ coming from (\ref{eqn:distort}).
\end{proof}
\begin{proof}[Proof of Proposition \ref{prop:inclusions}]
    For (i), set $k\in\mathbb{N}$, $x\in X_B$ and set 
    \begin{equation}
    \label{eqn:anuli}
    A_k(x) := C_k(x)\setminus C_{k+1}(x)\;\; \mbox{ and }\;\; A_0(x) := X_B\setminus\left( \{x\} \cup \bigcup_{k>0} A_k(x)\right).
    \end{equation}
    As such we have the disjoint union of sets
    $$X_B\setminus \{x\} =  \bigsqcup_{k\geq0 } A_k(x)$$
    and so the Dirichlet form can be written as
    $$\mathcal{E}^\psi _\gamma(u,u) = \frac12 \int_{X_B} \sum_{k\geq 0} \int_{A_k}\frac{(u(x)-u(y))^2}{d(x,y)^\gamma}\, d\mu_\psi(x)\, d\mu_\psi(y).$$
    Focusing on the integrand over $A_k$ and using Lemma \ref{lem:cnsts}:
    \begin{align*}
        \int_{A_k}\frac{(u(x)-u(y))^2}{d(x,y)^\gamma}\, d\mu_\psi(x)&  = \lambda^{k\gamma}\int_{A_k}\left(\sum_{i>k}u_i(x)-u_i(y)\right)^2\, d\mu_\psi(x) \\
        &\leq \lambda^{k\gamma}\int_{A_k}\left(\sum_{i>k}2\|u_i\|_\infty\right)^2\, d\mu_\psi(x) \\
        &= 4\lambda^{k\gamma}\mu_\psi(A_k)\left(\sum_{i>k}\|u_i\|_\infty\right)^2 \\
        &\leq 4\lambda^{k\gamma}C_{\mu_\psi} \lambda^{-kd_\psi}\left(\sum_{i>k}C_{u,\varepsilon}\lambda^{-i(r-\varepsilon)}\right)^2 \\
        & \leq C_{\mu_\psi,u,r, \varepsilon}\lambda^{k\gamma-kd_\psi -2k(r-\varepsilon)} = C_{\mu_\psi,u,r, \varepsilon}\lambda^{k(\gamma-d_\psi -2r+2\varepsilon)}.
    \end{align*}
    Thus the sum over the disjoint integrands $A_k$ converges as long as $\gamma-d_\psi-2r+2\varepsilon<0$, or $r>\frac{\gamma-d_\psi+\varepsilon}{2}$, proving the claim.

    Now (ii). Note that for any $f\in L^2$,
    $$\delta_k f = \sum_{e\in P_k} \alpha(e) \chi_e$$
    where the $\alpha(e)\in\mathbb{R}$ are the coefficients. Since $\delta_k f$ is a particular type of locally constant function in that there is a path $e'\in P_k$ such that $\|\delta_kf \|_\infty = |\alpha(e')|$, and so
    $$\|\delta_k f\|_\infty^2\mu_\psi(C_{e'}) = \|\delta_k f\|_{L^2_{\mu_\psi}(C_{e'})}^2 \;\;\; \mbox{ or }  \;\;\;\|\delta_k f\|_\infty \leq C_\psi \lambda^{\frac{k}{2}d_\psi}\|\delta_k f\|_{L^2_{\mu_\psi}(C_{e'})}.$$
    Since $C_{e_1}\cap C_{e_2} = \varnothing$ for $e_1\neq e_2\in P_k$, 
    $$\|\delta_k f\|_{L^2_{\mu_\psi}}^2 = \sum_{e\in P_k}\|\delta_k f\|_{L^2_{\mu_{\psi}}(C_e)}^2 \hspace{.7in}\mbox{  and so  } \hspace{.7in}\|\delta_k f\|_\infty \leq C \lambda^{\frac{k}{2}d_\psi}\|\delta_k f\|_{L^2_{\mu_\psi}}.$$
    It follows that
    \begin{align*}
    \|f\|_r = \sum_{k> 0} \lambda^{rk} \|\delta_k f\|_\infty &\leq C\sum_{k>0} \lambda^{\left(r+\frac{d_\psi}{2}\right)k}\|\delta_k f \|_{L^2} \\ 
    &\leq C\left( \sum_{k>0}\lambda^{-\varepsilon k}\right) ^\frac12 \left(\sum_{k>0} \lambda^{\left(2r+d_\psi +\varepsilon \right)k}\|\delta_k f \|_{L^2}^2 \right)^\frac12 \\
    &\leq C_\varepsilon \mathcal{E}^\psi _{2r + 2d_\psi + \varepsilon}(f,f)^\frac12 < \infty
    \end{align*}
    by Lemma \ref{lem:L2W}, where Cauchy-Schwartz was used in the second inequality.
\end{proof}

\section{Spectral properties of $\triangle_\gamma^\psi$}
\label{sec:spectral}
By the standard theory of Dirichlet forms \cite{FOT:DF}, there is a non-negative definite, unbounded self-adjoint operator $-\triangle_\gamma^\psi$ on $W_{\psi,\gamma}$ such that
$$\mathcal{E}^\psi _\gamma(u,v) = -(\triangle_\gamma^\psi u, v)$$
for all $v\in W_{\psi,\gamma}$ and $u\in \mathrm{Dom}(\triangle_\gamma^{\psi})\subset W_{\psi,\gamma} $. This section covers the proof of Theorem \ref{thm:spectrum}, the spectral theorem for the  Laplacian $\triangle^\psi_\gamma$ corresponding to $\mathcal{E}^\psi_\gamma $. For $k\geq 0$ and $e\in P_k$, recall the special subspaces of the locally constant functions
$$Y_\psi(e) = \left\{ f=\sum_{e'\in s^{-1}(r(e))} \alpha(e')\chi_{ee'}: \int f\, d\mu_\psi = \sum_{e'\in s^{-1}(r(e))} \mu_\psi(C_{ee'})\alpha(e') = 0\right\}.$$

First, a useful lemma, the proof of which can be found in Appendix \ref{app:B}.
\begin{lemma}
\label{lem:polar-cylinder}
Let $e\in P_k$. For any real $\{a_{ee_i}\}_{i=1}^{m(e)}$ and $\{b_{ee_i}\}_{i=1}^{m(e)}$, set
\[
a_e:=\frac{1}{\mu_e}\sum_{i=1}^{m(e)}\mu_{ee_i}\,a_{ee_i},\qquad
b_e:=\frac{1}{\mu_e}\sum_{i=1}^{m(e)}\mu_{ee_i}\,b_{ee_i}.
\]
Then
\begin{equation}
\label{eq:polar-cylinder}
\sum_{1\le i<j\le m(e)} \mu_{ee_i}\mu_{ee_j}\,\big(a_{ee_i}-a_{ee_j}\big)\big(b_{ee_i}-b_{ee_j}\big)
\;=\;
\mu_e\,\sum_{i=1}^{m(e)} \mu_{ee_i}\,\big(a_{ee_i}-a_e\big)\big(b_{ee_i}-b_e\big).
\end{equation}
\end{lemma}
Let $\psi$ be H\"older continuous and $\mu_\psi$ its unique Gibbs state. In what follows denote by $\Pi_k^\psi:L^2_{\mu_\psi}\rightarrow L^2_{\mu_\psi}$ the families of conditional expectations defined in \S \ref{subsec:spaces}.
\begin{lemma}
\label{lem:diag-cylinder}
Fix $e\in P_k$ and $u\in Y_\psi(e)$. For any $v\in L^2_{\mu_\psi}$,
\begin{equation}
\label{eq:diag-cylinder}
\mathcal{E}^\psi_\gamma(u,v)
=
\Bigg(\frac{\mu_\psi(C_e)}{\operatorname{diam}(C_e)^\gamma}
+\int_{X_B\setminus C_e}\frac{d\mu_\psi(y)}{d(C_e,y)^\gamma}\Bigg)\int_{C_e} u\,v\,d\mu_\psi.
\end{equation}
\end{lemma}

\begin{proof}
Write $u|_{C_{ee_i}}=u_{ee_i}$ and $v_{ee_i}:=\Pi^\psi_{k+1}v(z_{ee_i})$, where $z_{ee_i}\in C_{ee_i}$. Since $u$ is supported on $C_e$, the integrals associated to $\mathcal{E}^\psi_\gamma (u,v)$ are supported on the set 
$$(C_e\times C_e)\sqcup (C_e\times(X_B\setminus C_e) ) \sqcup ((X_B\setminus C_e)\times C_e).$$
Denote $\mathcal{E}^\psi_\gamma  = I_{\mathrm{int}}+I^{(1)}+I^{(2)}$ according to the domains of integration above.

For $I_{\mathrm{int}}$, note that for $x\in C_{ee_i}$, $y\in C_{ee_j}$ with $i\neq j$ we have $d(x,y)=\operatorname{diam}(C_e)$,
and $u(x)-u(y)=u_{ee_i}-u_{ee_j}$. As such,
\begin{align*}
I_{\mathrm{int}}
&=\frac{1}{2\,\operatorname{diam}(C_e)^\gamma}
\sum_{i\neq j}(u_{ee_i}-u_{ee_j})\!\!\iint_{C_{ee_i}\times C_{ee_j}}\!\!(v(x)-v(y))\,d\mu(x)d\mu(y) \\
&=\frac{1}{\operatorname{diam}(C_e)^\gamma}\sum_{i<j}\mu_\psi(C_{ee_i})\mu_\psi(C_{ee_j})(u_{ee_i}-u_{ee_j})(v_{ee_i}-v_{ee_j}),
\end{align*}
since $\int_{C_{ee_i}}v\,d\mu_\psi=\mu_\psi(C_{ee_i})v_{ee_i}$. Using Lemma \ref{lem:polar-cylinder} with
$a_{ee_i}=u_{ee_i}$ and $b_{ee_i}=v_{ee_i}$ and that $\mu_\psi(u)=0$ (since $u\in Y_\psi(e)$), it follows that
\[
I_{\mathrm{int}}
=\frac{\mu_\psi(C_e)}{\operatorname{diam}(C_e)^\gamma}\sum_{i=1}^{m(e)} \mu_\psi(C_{ee_i})\,u_{ee_i}\,v_{ee_i}
=\frac{\mu_\psi(C_e)}{\operatorname{diam}(C_e)^\gamma}\int_{C_e} u\,(\Pi^\psi_{k+1}v)\,d\mu_\psi.
\]

Since $u$ is constant on each $C_{ee_i}$,
\[
\int_{C_e} u\,(v-\Pi^\psi_{k+1}v)\,d\mu_\psi
=\sum_i u_{ee_i}\Big(\int_{C_{ee_i}} v\,d\mu_\psi-\mu_\psi(C_{ee_i})v_{ee_i}\Big)=0,
\]
and thus
\begin{equation}\label{eq:replace-v-by-avg}
\int_{C_e} u\,(\Pi^\psi_{k+1}v)\,d\mu_\psi=\int_{C_e} u\,v\,d\mu_\psi.
\end{equation}
Therefore
\[
I_{\mathrm{int}}=\frac{\mu_\psi(C_e)}{\operatorname{diam}(C_e)^\gamma}\int_{C_e} u\,v\,d\mu_\psi.
\]

For, $I^{(1)}+I^{(2)}$,
for $x\in C_e$ and $y\in X_B\setminus C_e$, the extension of $u$ by $0$ outside $C_e$ gives
$u(y)=0$, and ultrametricity gives $d(x,y)=d(C_e,y)$, independent of $x$. So $I^{(1)}$ is
$$I^{(1)}=\frac12 \int_{X_B\setminus C_e} \frac{d\mu_\psi(y)}{d(C_e,y)^\gamma}\int_{C_e}u v\,d\mu_\psi
   \ -\ \frac12\left(\int_{X_B\setminus C_e}\frac{v(y)}{d(C_e,y)^\gamma}\,d\mu_\psi(y)\right)\int_{C_e}u\,d\mu_\psi,$$
and similarly for $I^{(2)}$. Using that $\mu_\psi(u) = \int_{C_e}u\, d\mu_\psi = 0$, it follows that
\begin{align*}
I^{(1)}
&=\frac12\int_{X_B\setminus C_e}\frac{d\mu_\psi(y)}{d(C_e,y)^\gamma}\int_{C_e} u(x)v(x)\,d\mu_\psi(x) , \hspace{.3in}\mbox{and}\\ 
I^{(2)}
&=\frac12\int_{X_B\setminus C_e}\frac{d\mu_\psi(x)}{d(C_e,x)^\gamma}\int_{C_e} u(y)v(y)\,d\mu_\psi(y),
\end{align*}
and so
\[
I^{(1)}+I^{(2)}
=\Bigg(\int_{X_B\setminus C_e}\frac{d\mu_\psi(y)}{d(C_e,y)^\gamma}\Bigg)\int_{C_e} u\,v\,d\mu_\psi.
\]
Putting everything together gives \eqref{eq:diag-cylinder}.
\end{proof}

\begin{corollary}
\label{cor:eig-cylinder}
Let $\triangle^\psi_\gamma$ be the self–adjoint operator on $L^2_{\mu_\psi}$ associated with $\mathcal{E}^\psi_\gamma $.
For each $e\in P_k$ and $u\in Y_\psi(e)$,
\begin{equation}
\label{eqn:evalue}
\triangle^\psi_\gamma u=\lambda(e)\,u\quad\text{in }L^2(\mu),
\qquad
\lambda_\psi(e):=\frac{\mu_\psi(C_e)}{\operatorname{diam}(C_e)^\gamma}
+\int_{X_B\setminus C_e}\frac{d\mu_\psi(y)}{d(C_e,y)^\gamma}.
\end{equation}
In particular, $\lambda_\psi(e)$ has multiplicity $m(e)=\dim Y_\psi(e)$.
\end{corollary}
\begin{remark}
    The expression (\ref{eqn:evalue}) is a version of the Khrennikov-Kozyrev wavelet eigenvalue formula found in \cite{KK:wavelets}.
\end{remark}
\begin{lemma}
\label{lem:ancestor}
For $\tilde e\in P_{k-1}$ and $e = \tilde{e}e' \in P_k$ (i.e. $C_e\subset C_{\tilde e}$), then
\[
\int_{X_B\setminus C_e}\frac{d\mu_\psi(y)}{d(C_e,y)^\gamma}
=\int_{X_B\setminus C_{\tilde e}}\frac{d\mu_\psi(y)}{d(C_{\tilde e},y)^\gamma}
+\frac{\mu_\psi(C_{\tilde e}\setminus C_e)}{\operatorname{diam}(C_{\tilde e})^\gamma}.
\]
\end{lemma}

\begin{proof}
Write $X_B\setminus C_e=(X_B\setminus C_{\tilde e})\sqcup (C_{\tilde e}\setminus C_e)$. If $y\in X_B\setminus C_{\tilde e}$ then
$d(C_e,y)=d(C_{\tilde e},y)$. Likewise, if $y\in C_{\tilde e}\setminus C_e$ then $d(C_e,y)=\operatorname{diam}(C_{\tilde e})$. Integrate over the two pieces.
\end{proof}

\begin{corollary}
\label{cor:gap-cylinder}
If $\tilde e \in P_{k-1}$ and $e = \tilde{e}e'\in P_k$, then
\[
\lambda_\psi(e)-\lambda_\psi(\tilde e)
=\mu_\psi(C_e)\Big(\operatorname{diam}(C_e)^{-\gamma}-\operatorname{diam}(C_{\tilde e})^{-\gamma}\Big)\ >\ 0.
\]
Thus the smallest positive eigenvalue is attained at level $0$:
\[
\lambda_1=\sum_{v\in V_0}\frac{\mu_\psi(C_v)}{\operatorname{diam}(X_B)^\gamma}
=\frac{\mu_\psi(X_B)}{\operatorname{diam}(X_B)^\gamma},
\]
with eigenspace of dimension $|V_0|-1$.
\end{corollary}

\begin{proof}
By Corollary \ref{cor:eig-cylinder} and Lemma \ref{lem:ancestor},
\begin{align*}
\lambda_\psi(e)-\lambda_\psi(\tilde e)
&=\frac{\mu_\psi(C_e)}{\operatorname{diam}(C_e)^\gamma}-\frac{\mu_\psi(C_{\tilde e})}{\operatorname{diam}(C_{\tilde e})^\gamma}
+\Big(\int_{X_B\setminus C_e}\frac{d\mu_\psi}{d(C_e,\cdot)^\gamma}-\int_{X_B\setminus C_{\tilde e}}\frac{d\mu_\psi}{d(C_{\tilde e},\cdot)^\gamma}\Big)\\
&=\frac{\mu_\psi(C_e)}{\operatorname{diam}(C_e)^\gamma}-\frac{\mu_\psi(C_{\tilde e})}{\operatorname{diam}(C_{\tilde e})^\gamma}
+\frac{\mu_\psi(C_{\tilde e})-\mu_\psi(C_e)}{\operatorname{diam}(C_{\tilde e})^\gamma}\\
&=\mu_\psi(C_e)\Big(\operatorname{diam}(C_e)^{-\gamma}-\operatorname{diam}(C_{\tilde e})^{-\gamma}\Big)>0,
\end{align*}
since $\operatorname{diam}(C_e)<\operatorname{diam}(C_{\tilde e})$. The formula for $\lambda_1$ and its eigenspace follows by taking $e$ at level $0$ (the partition $\{C_v\}_{v\in V_0}$) in Corollary \ref{cor:eig-cylinder}.
\end{proof}
    For the Poincar\'e inequality, let $f$ have zero average. Then by the bounds on eigenvalues and (\ref{eqn:L2norm}):
    $$\mathcal{E}^\psi_\gamma (f,f) = \sum_{k>0} \sum_{e\in P_k}\lambda_\psi(e)\|f|_{C_e}\|^2_{L^2(C_e)}\geq \lambda_1 \sum_{k>0} \sum_{e\in P_k}\|f\|^2_{L^2_{\mu_\psi}(C_e)} = \lambda_1\|f\|^2_{L^2_{\mu_\psi}}$$
    which proves the Poincar\'e inequality. What remains to investigate is the Weyl law. 
    \begin{proposition}
    \label{prop:Weyl}
        For $\gamma>d_\psi$ there is a constant $C_{\psi,\gamma}$ such that for all $\Lambda>1$
        $$C^{-1}_{\psi,\gamma} \Lambda^{\frac{1}{\gamma - d_\psi}} \leq  N^\psi_\gamma(\Lambda)\leq C_{\psi,\gamma} \Lambda^{\frac{1}{\gamma - d_\psi}}.$$
    \end{proposition}
    \begin{proof}
        Define
        $$M_k:= \sum_{e\in P_k}m(e),$$
        where $m(e):=\mathrm{dim}(Y_\psi(e)) = |s^{-1}(r(e))|-1$ is the multiplicity of $\lambda_\psi(e)$, that is, 
        $M_k$ is the number of eigenvalues contributed by paths in $P_k$. In what follows, denote $A\sim B$ if there is a $K>1$ such that $K^{-1}B\leq A \leq K B $.
        
        By Lemma \ref{lem:diag-cylinder},
        $$\lambda_\psi(e) = \frac{\mu_\psi(C_e)}{\operatorname{diam}(C_e)^\gamma} +\int_{X_B\setminus C_e}\frac{d\mu_\psi(y)}{d(C_e,y)^\gamma}$$
    and so by (\ref{eqn:distort}) there is a constant $C$ such that 
    $$C^{-1} \lambda^{(\gamma-d_\psi)k}+\int_{X_B\setminus C_e}\frac{d\mu_\psi(y)}{d(C_e,y)^\gamma}\leq \lambda_\psi(e) \leq C \lambda^{(\gamma-d_\psi)k}+\int_{X_B\setminus C_e}\frac{d\mu_\psi(y)}{d(C_e,y)^\gamma}.$$
    Now, recalling (\ref{eqn:anuli}), there is a decomposition of $X_B\setminus C_e$ as
    $$X_B\setminus C_e = \left(X_B\setminus C_{s(e)}\right)\sqcup\bigsqcup_{i=0}^{k-1} A_i(x_e)$$
    where $x_e \in C_e$ and where $A_i(x) = C_i(x)\setminus C_{i+1}(x)$. On each $A_i(x_e)$, the distance to $C_e$ is constant and equal to the diameter of the set, so it follows that
    \begin{equation}
    \label{eqn:breakdown}
    \int_{X_B\setminus C_e}\frac{d\mu_\psi(y)}{d(C_e,y)^\gamma} = \frac{\mu_\psi(X_B\setminus C_{s(e)})}{\mathrm{diam}(X_B)^\gamma} + \sum_{i = 0 }^{k-1}\frac{\mu_\psi(A_i(x_e))}{\mathrm{diam}(C_i(x_e))^\gamma}.
    \end{equation}
    Now, using the fact that $\mathrm{diam}(C_i(x_e))$ scales like $\lambda^{-i}$, it follows that $\lambda_\psi(e)\sim \lambda^{(\gamma-d_\psi)k}$. 
    
    Let $C_*>1$ be a number such that $C_*^{-1} \lambda^{(\gamma-d_\psi)k}\leq \lambda_\psi(e)\leq C_* \lambda^{(\gamma-d_\psi)k}$ for all $e\in P_k$ and $k\geq 0$. For $\Lambda\geq 1$, denote by $k(\Lambda)$ the unique non-negative integer such that
    \begin{equation}
    \label{eqn:eigCount}
    C^{-1}_*\lambda^{(\gamma-d_\psi)k(\Lambda)}\leq \Lambda < C^{-1}_*\lambda^{(\gamma-d_\psi)(k(\Lambda)+1)}.
    \end{equation}
    This means that all eigenvalues from levels less than or equal to $k(\Lambda)$ are less than or equal to $\Lambda$ while the eigenvalues from levels greater than $k(\Lambda)$ are greater than $\Lambda$. As such,
    $$\sum_{i = 0}^{k(\Lambda)} M_i \leq N(\Lambda) \leq \sum_{i = 0}^{k(\Lambda)+1} M_i  .$$
    Since $B$ is stationary, there exists a constant $K>1$ such that
    $K^{-1} \lambda^{k}\leq M_k \leq K\lambda^{k}$
    for all $k\in\mathbb{N}$ (this follows from the Perron-Frobenius theorem). Thus, there is a constant $K_*>1$ such that $  N(\Lambda)\sim  \lambda^{k(\Lambda)}$.
    Finally, recalling (\ref{eqn:eigCount}), it follows that
    $$k(\Lambda) \sim \frac{\log(\Lambda)}{(\gamma-d_\psi)\log\lambda}$$
    and so $N(\Lambda)\sim \Lambda^{\frac{1}{\gamma-d_\psi}}$.    
    \end{proof}

Now onto two-sided estimates for the heat kernel associated to the nonlocal Dirichlet form and its generator $\triangle_\gamma^\psi$. Let $T_t^{\psi,\gamma}=e^{t\triangle^\psi_\gamma}$ be the heat semi-group operator generated by $\triangle^\psi_\gamma$ and denote by $p_t^{\psi,\gamma}$ its integral kernel, i.e., the heat kernel.
\begin{proposition}
The integral kernel of the heat operator $T_t^{\psi,\gamma}$ for $t\in (0,1]$ satisfies the two-sided estimate
$$c_1 t^{-\frac{d_\psi}{\gamma-d_\psi}}\left(1+\frac{d(x,y)}{t^{1/(\gamma-d_\psi)}}\right)^{-\gamma}\leq p_t^{\psi,\gamma}(x,y) \leq    c_2 t^{-\frac{d_\psi}{\gamma-d_\psi}}\left(1+\frac{d(x,y)}{t^{1/(\gamma-d_\psi)}}\right)^{-\gamma}$$
for some $c_1,c_2>0$. Moreover, it is continuous jointly in $x,y,t$ and H\"older continus in $x,y$.
\end{proposition}

\begin{proof}
This follows from  (\ref{eqn:distort}) and \cite[Corollary 2.13]{BGHH:HK}.
\end{proof}

\subsection{Examples}
\label{subsec:examples}
This section goes over computations for the spectrum of $\triangle_\gamma = \triangle^0_\gamma$ defined on $X_B$ for three related stationary Bratteli diagrams, where the zero potential $\psi = 0$ is taken in order to study the measure of maximal entropy with respect to the shift which is the same as the unique tail-invariant probability measure which will be denoted by $\mu$. The three stationary diagrams are given by the following matrices:
$$A_1 = (2),\;\;\;A_2 = \begin{bmatrix} 1&1 \\ 1&1\end{bmatrix},\;\;\;A_3 = (4)$$
and they are depicted in Figure \ref{fig:diagrams}. The reason these matrices are chosen is that they all define Bratteli diagrams with isomorphic topological invariants and conjugate dynamics. Thus it is worthwhile to see whether how their spectral invariants compare. Recall that our choices of metric and measure are such that $\mu(X_B)=\operatorname{diam}(X_B)=1$.

\subsubsection{$A_1$}
\label{subsubsec:simplest}
The Bratteli diagram defined by $A_1$ has one vertex at every level and two edges at every level. Moreover, if $w\in P_k$, then $\mu(C_e) = 2^{-k}$ for the unique tail-invariant probability measure $\mu$. Moreover, for the sets $C_i(x)$ satisfy $\mathrm{diam}(C_i(x)) = 2^{-i}$.
Thus, by (\ref{eqn:evalue}) and (\ref{eqn:breakdown}), for $e\in P_k$,
\begin{equation*}
\begin{split}
\lambda(e) &= \frac{\mu(C_e)}{\mathrm{diam}(C_e)^\gamma}+\int_{X_B\setminus C_e} \frac{d\mu(y)}{d(C_e,y)^\gamma}  = 2^{(\gamma-1)k}+ \sum_{i=0}^{k-1} \frac{2^{-i}-2^{-(i+1)}}{2^{-i\gamma}}\\
& = 2^{(\gamma-1)k}+ \frac12 \sum_{i=0}^{k-1} 2^{i(\gamma-1)} = 2^{(\gamma-1)k}+ \frac12 \frac{1-2^{(\gamma-1)k}}{1-2^{\gamma-1}}.
\end{split}
\end{equation*}
Thus, when $\gamma\neq 1$ the spectrum is the set
\begin{equation}
\label{eqn:spec1}
    \sigma_1:=\sigma(\triangle_\gamma) = \begin{cases}
        \left\{ 2^{(\gamma-1)k}+ \frac12 \frac{1-2^{(\gamma-1)k}}{1-2^{\gamma-1}} : k\in\mathbb{N}\cup \{0\}\right\} &\mbox{ if }\gamma\neq 1\\
        \left\{ 1+\frac{k}{2}:k\in\mathbb{N} \right\}&\mbox{ if }\gamma = 1 .
    \end{cases}
\end{equation}

Consider the Bratteli diagram defined by the Pascal graph \cite{Vershik:Pascal}, which is not a stationary diagram. It is straight forward to check that there is a natural measure on its path space, analogous to the measure of maximal entropy, and a natural ultrametric such that the spectrum $\sigma_{Pascal} = \sigma(\triangle_\gamma)$ satisfies $\sigma_{Pascal}= \sigma_1$, yet the diagrams are far from being the same. Significant differences are that the diagram defined by $A_1$ has a unique tail-invariant probability measure and its cohomology is the smallest it can be (one-dimensional), whereas the Pascal graph Bratteli diagram has a continuum of tail-invariant probability measures and has infinite dimensional cohomology.

\subsubsection{$A_2$}
Consider now the Bratteli diagram given by the matrix $A_2$. This diagram consists of two vertices at every level, four edges at every level, and between levels there is a single edge connecting any two vertices on consecutive levels.

For the unique tail-invariant probability measure $\mu$ on $X_B$ and any $e\in P_k$ we have $\mu(C_e) = 2^{-(k+1)}$. Unlike the previous example, since $V_0$ has two vertices in this case, the first term in the right hand side of (\ref{eqn:breakdown}) is nonzero and equal to $\frac12$. So for $e\in P_k$,
\begin{equation*}
\begin{split}
\lambda(e) &= \frac{\mu(C_e)}{\mathrm{diam}(C_e)^\gamma}+\int_{X_B\setminus C_e} \frac{d\mu(y)}{d(C_e,y)^\gamma}  = \frac12 2^{(\gamma-1)k}+ \frac12 +\sum_{i=0}^{k-1} \frac{2^{-i-1}-2^{-i-2}}{2^{-i\gamma}}\\
& = \frac12 2^{(\gamma-1)k}+\frac12 + \frac14 \sum_{i=0}^{k-1} 2^{(\gamma-1)i} = \frac12 2^{(\gamma-1)k}+ \frac12 + \frac{1}{4}\frac{1-2^{(\gamma-1)k}}{1-2^{\gamma-1}}.
\end{split}
\end{equation*}
Thus the spectrum is 
\begin{equation}
\label{eqn:spec2}
    \sigma_2:=\sigma(\triangle_\gamma) = \begin{cases}
        \left\{ \frac12 2^{(\gamma-1)k}+ \frac12 + \frac{1}{4}\frac{1-2^{(\gamma-1)k}}{1-2^{\gamma-1}} : k\in\mathbb{N}\cup\{0\}\right\} &\mbox{ if }\gamma\neq 1\\
        \left\{ 1+\frac{k}{4}:k\in\mathbb{N} \right\}&\mbox{ if }\gamma = 1 .
    \end{cases}
\end{equation}
and note that if $\gamma \neq 1$, $\sigma_2  = \frac12 (\sigma_1+ 1)$.
\subsubsection{$A_3$}
The Bratteli diagram defined by $A_3$ has one vertex at every level and four edges at every level. Moreover, if $w\in P_k$, then $\mu(C_e) = 4^{-k} = 2^{-2k}$ for the unique tail-invariant probability measure $\mu$. Thus, by (\ref{eqn:evalue}) and (\ref{eqn:breakdown}), for $e\in P_k$,
\begin{equation*}
\begin{split}
\lambda(e) &= \frac{\mu(C_e)}{\mathrm{diam}(C_e)^\gamma}+\int_{X_B\setminus C_e} \frac{d\mu(y)}{d(C_e,y)^\gamma}  = 4^{(\gamma-1)k}+ \sum_{i=0}^{k-1} \frac{4^{-i}-4^{-(i+1)}}{4^{-i\gamma}}\\
& = 4^{(\gamma-1)k}+ \frac{3}{4}\sum_{i=0}^{k-1} 4^{i(\gamma-1)} = 4^{(\gamma-1)k}+ \frac{3}{4}\frac{1-4^{(\gamma-1)k}}{1-4^{\gamma-1}}.
\end{split}
\end{equation*}
Thus, the spectrum is the set
\begin{equation}
\label{eqn:spec3}
    \sigma_3:=\sigma(\triangle_\gamma) = \begin{cases}
        \left\{ 4^{(\gamma-1)k}+ \frac{3}{4}\frac{1-4^{(\gamma-1)k}}{1-4^{\gamma-1}} : k\in\mathbb{N}\cup \{0\}\right\} &\mbox{ if }\gamma\neq 1\\
        \left\{ 1+\frac{3}{4}k:k\in\mathbb{N} \right\}&\mbox{ if }\gamma = 1 .
    \end{cases}
\end{equation}
\section{A cohomological Dirichlet principle for Cantor sets}
\label{sec:Hodge}
This section proves Theorem \ref{thm:hodge}. First, the domains of the Dirichlet forms need to be connected to spaces with well-defined cohomology.
\begin{lemma}
\label{lem:extend}
     For $\gamma>2(1+d_\psi - \frac{\log\lambda_-}{\log\lambda})$ the distributions $\mathcal{D}_\tau$ extend to continuous functionals on $W_{\psi,\gamma}$.
 \end{lemma}
 \begin{proof}
 By Proposition \ref{prop:inclusions} there is a continuous embedding $i:W_\gamma \hookrightarrow \mathcal{S}_{r_*}$  for some $r_* > 1-\frac{\log\lambda_-}{\log\lambda}$, and so the functionals $\mathcal{D}_\tau$ can be pulled back to give continuous functionals by Theorem \ref{thm:extension}.
 \end{proof}

For $\gamma> 2(1+d_\psi - \frac{\log\lambda_-}{\log\lambda})$, let
$$\mathcal{B}_\gamma^\psi := \left\{ f\in W_{\psi,\gamma} : \mathcal{D}_\tau(f) = 0 \mbox{ for all }\tau\in \mathrm{Tr}(LF(A))  \right\}$$
and denote by
$$H^\psi _\gamma(X_B):= W_{\psi,\gamma} /\mathcal{B}^\psi_\gamma .$$

Since Proposition \ref{prop:inclusions} gives the relationship between $r$ and $\gamma$ such that $W_{\psi,\gamma} \subset \mathcal{S}_r$ and Theorem \ref{thm:DirLim} gives the conditions needed on $r$ for the cohomology to be finite dimensional, putting these together the following holds.
\begin{corollary}
If $\gamma> 2(1+d_\psi - \frac{\log\lambda_-}{\log\lambda}) $, then the cohomology $H_\gamma^\psi (X_B)$, is finite-dimensional and isomorphic to $H_{lc}^0(X_B)$.
\end{corollary}

\begin{theorem}
\label{thm:harmonic}
Let $X_B$ be the path space of a stationary simple Bratteli diagram, let $\mu_\psi$ be the unique Gibbs measure for $\psi$. Consider the non–local Dirichlet form $\mathcal{E}^\psi_\gamma$ with domain $W_{\psi,\gamma}$ and $\gamma>2(1+d_\psi - \frac{\log\lambda_-}{\log\lambda})$. For $f\in W_{\psi,\gamma}^0$ let the cohomology class be $[f]\in H_\gamma^\psi(X_B)$, and set
$$J([f]):=\inf_{g\in[f]}\mathcal{E}^\psi_\gamma(g,g).$$
Then for every $f\in W_{\psi,\gamma}^0$ there exists a unique $h\in [f]$ of zero average such that
\begin{itemize}
\item $\mathcal{E}^\psi_\gamma(h,b)=0$ for all $b\in\mathcal{B}^\psi_\gamma$;
\item $h$ minimizes the energy in its class: $\mathcal{E}^\psi_\gamma(h,h)=J([f])$.
\end{itemize}
As such, every class $[f]$ contains a unique $\mathcal{E}_\gamma^{\psi}$-minimizing representative $h$.
\end{theorem}

\begin{proof}
The proof is straight-forward and goes through the so-called direct method in the calculus of variations. For $f\in W_{\psi,\gamma}^0 $ with $[f]\neq 0 \in H^\psi_\gamma(X_B)$, consider
$$J([f]) = \inf_{g\in [f]}\mathcal{E}^\psi_\gamma (g,g).$$
Let $g_i$ be a sequence in $W_\gamma$ such that $\mathcal{E}^\psi_\gamma  (g_i,g_i)\rightarrow J([f])$. By the Poincar\'e inequality in Theorem \ref{thm:spectrum}, this is a bounded sequence in $W_{\psi,\gamma}$, and so there is a weakly convergent subsequence $g_i\rightarrow g^*$.

Since each $\mathcal{D}_\tau$ is bounded on $W^0_{\psi,\gamma}$, it is weakly continuous. Because $\mathcal{D}_\tau(g_i)=\mathcal{D}_\tau(f)$ for all $\tau \in \mathrm{Tr}(LF(A))$, we obtain $\mathcal{D}_\tau(g^*)=\lim_i \mathcal{D}_\tau(g_i)=\mathcal{D}_\tau(f)$, i.e. $g^*\in [f]$.
Moreover, $\mathcal{E}^\psi_\gamma (\cdot,\cdot)$ is weakly lower semicontinuous on $W_\gamma^0 $: indeed, from $\mathcal{E}^\psi_\gamma(u-v,u-v)\geq 0$ it follows that
$$\mathcal{E}^\psi_\gamma(u,u) = \sup_{v\in W_{\psi,\gamma}^0}\left\{ 2\mathcal{E}^\psi_\gamma(u,v) - \mathcal{E}^\psi_\gamma(v,v)\right\}$$
with equality at $u=v$. Since the right-hand side is a pointwise supremum of weakly continuous affine functionals in $u$, it is weakly lower semicontinuous. As such,
$$ \mathcal{E}^\psi_\gamma (g^*,g^*)\ \le\ \liminf_{i\to\infty}\ \mathcal{E}^\psi_\gamma (g_i,g_i)\ =\ J([f]),
$$
and since $g^*\in [f]$, we must have $\mathcal{E}^\psi_\gamma (g^*,g^*)=J([f])$.
Thus $g^*$ is a minimizer.

Now let $b\in\mathcal{B}^\psi_\gamma$. For each $t\in\mathbb{R}$ we have $g^*+t b\in [f]$.
Define $\phi(t):=\mathcal{E}^\psi_\gamma (g^*+t b,g^*+t b)$.
Then $\phi$ is a quadratic polynomial with $\phi'(0)=2\,\mathcal{E}^\psi_\gamma (g^*,b)$.
Since $t=0$ minimizes $\phi$ on $\mathbb{R}$, we obtain $\mathcal{E}^\psi_\gamma (g^*,b)=0$ for all $b\in\mathcal{B}^\psi_\gamma$. 

To show uniqueness, suppose $f_1,f_2\in W^0_{\psi,\gamma}$ are two minimizing representatives of the same class. Then $f_1-f_2 \in\mathcal{B}_\gamma^\psi$ and
$$\mathcal{E}^\psi _\gamma(f_1-f_2,f_1-f_2) = \mathcal{E}^\psi _\gamma(f_1,f_1-f_2) - \mathcal{E}^\psi _\gamma(f_2,f_1-f_2) = 0-0$$
 by the previous paragraph, and thus $f_1 - f_2  =  c$, a constant. But since they both have zero average (we have $f_i\in W^0_{\psi,\gamma}$), $c = 0$ and $f_1 = f_2$. 
\end{proof}

 \appendix
 \section{Derivation of (\ref{eqn:app})}
 \label{app:tedious}
 To avoid extra notational cluttering, denote $\mu_e:= \mu(C_{e})$. The following equalities which will be used in the derivation:
$$\sum_{e'\in s^{-1}(r(e))} \mu_{ee'}  = \mu_e \hspace{.5in} \mbox{ and } \hspace{.5in}\sum_{e'\in s^{-1} (r(e))}\mu_{ee'} \Pi_{k+1}f(z_{ee'})  =  \mu_e\Pi_kf(z_e).$$
Thus for $e\in P_k$, in baby steps:
\begin{align*}
    &\frac12 \sum_{\overset{e_1,e_1\in s^{-1}(r(e))}{e_1\neq e_2}}\mu_{ee_1}\mu_{ee_2}\left(\Pi_{k+1}f(z_{ee_1})- \Pi_{k+1}f(z_{ee_2})\right)^2 \\
    &\hspace{.2in} =\frac12 \sum_{\overset{e_1,e_1\in s^{-1}(r(e))}{e_1\neq e_2}}\mu_{ee_1}\mu_{ee_2}\left(\Pi_{k+1}f(z_{ee_1})^2+ \Pi_{k+1}f(z_{ee_2})^2 - 2 \Pi_{k+1}f(z_{ee_1})\Pi_{k+1}f(z_{ee_2})\right)  \\
    &\hspace{.2in} = \frac12 \left( \sum_{e'\in s^{-1}(r(e))} \mu_{ee'} \right) \sum_{e'\in s^{-1} (r(e))}\mu_{ee'} \Pi_{k+1}f(z_{ee'})^2 - \frac12 \left( \sum_{e'\in s^{-1} (r(e))}\mu_{ee'} \Pi_{k+1}f(z_{ee'}) \right)^2 \\
    &\hspace{.2in} = \mu_e \sum_{e'\in s^{-1} (r(e))}\mu_{ee'} \Pi_{k+1}f(z_{ee'})^2 - \left( \sum_{e'\in s^{-1} (r(e))}\mu_{ee'} \Pi_{k+1}f(z_{ee'}) \right)^2 \\
    &\hspace{.2in} = \mu_e \sum_{e'\in s^{-1} (r(e))}\mu_{ee'} \Pi_{k+1}f(z_{ee'})^2 - \mu_e^2\Pi_kf(z_e)^2 \\
    &\hspace{.2in} = \mu_e \sum_{e'\in s^{-1} (r(e))}\mu_{ee'} \Pi_{k+1}f(z_{ee'})^2 - 2\mu_e^2\Pi_kf(z_e)^2  + \mu_e^2\Pi_kf(z_e)^2 \\
    &\hspace{.2in} = \mu_e \sum_{e'\in s^{-1} (r(e))}\mu_{ee'} \Pi_{k+1}f(z_{ee'})^2 - 2\mu_e\Pi_kf(z_e)(\mu_e\Pi_kf(z_e))  + \mu_e\Pi_kf(z_e)^2 (\mu_e) \\
    &\hspace{.2in} = \mu_e \left(\sum_{e'\in s^{-1} (r(e))}\mu_{ee'} \Pi_{k+1}f(z_{ee'})^2 - 2\Pi_kf(z_e)(\mu_e\Pi_kf(z_e))  + \Pi_kf(z_e)^2 (\mu_e) \right) \\
    &\hspace{.2in} = \mu_e \left(\sum_{e'\in s^{-1} (r(e))}\mu_{ee'} \Pi_{k+1}f(z_{ee'})^2 - 2\Pi_kf(z_e)\sum_{e'\in s^{-1} (r(e))}\mu_{ee'} \Pi_{k+1}f(z_{ee'}) \right. \\
    &\hspace{3.5in}\left. + \Pi_kf(z_e)^2 \sum_{e'\in s^{-1}(r(e))} \mu_{ee'} \right) \\
    &\hspace{.2in} = \mu_e \left(\sum_{e'\in s^{-1} (r(e))}\mu_{ee'} \left( \Pi_{k+1}f(z_{ee'})^2 - 2\Pi_kf(z_e) \Pi_{k+1}f(z_{ee'})  + \Pi_kf(z_e)^2 \right) \right)  \\
    &\hspace{.2in} = \mu_e \sum_{e'\in s^{-1} (r(e))}\mu_{ee'} \left( \Pi_{k+1}f(z_{ee'}) -  \Pi_kf(z_{ee'}) \right)^2   \\
    &\hspace{.2in} = \mu_e \sum_{e'\in s^{-1} (r(e))}\mu_{ee'} \left( \delta_{k+1}f(z_{ee'}) \right)^2    = \mu_e \|\delta_{k+1} f \|^2_{L^2(C_e)}.
\end{align*}
\section{Proof of Lemma \ref{lem:polar-cylinder}}
\label{app:B}
To reduce notational tediousness, define first
$$
\mu_i:=\mu_{ee_i},\qquad \mu_e:=\sum_i \mu_i,\qquad
a_i:=a_{ee_i},\qquad b_i:=b_{ee_i},
$$
and define the weighted averages
$$
a_e:=\frac{1}{\mu_e}\sum_i \mu_i,a_i,\qquad
b_e:=\frac{1}{\mu_e}\sum_i \mu_i,b_i.
$$

Now rewrite $\sum_{i<j}$ as $\tfrac12\sum_{i\neq j}$ to get
$$
S:=\sum_{i<j}\mu_i\mu_j(a_i-a_j)(b_i-b_j)
=\frac12\sum_{i\neq j}\mu_i\mu_j(a_i-a_j)(b_i-b_j).
$$
Now expand the product:
$$
(a_i-a_j)(b_i-b_j)=a_i b_i + a_j b_j - a_i b_j - a_j b_i
$$
and write
\begin{equation*}
\begin{split}
2S
&=\sum_{i\neq j}\mu_i\mu_j a_i b_i
+\sum_{i\neq j}\mu_i\mu_j a_j b_j
-\sum_{i\neq j}\mu_i\mu_j a_i b_j
-\sum_{i\neq j}\mu_i\mu_j a_j b_i \\
&= (I) + (II)- (III) - (IV).
\end{split}
\end{equation*}

For (I),
$$
\sum_{i\neq j}\mu_i\mu_j a_i b_i
=\sum_i\mu_i a_i b_i\sum_{j\neq i}\mu_j
=\sum_i\mu_i a_i b_i,(\mu_e-\mu_i)
=\mu_e\sum_i\mu_i a_i b_i-\sum_i \mu_i^2 a_i b_i.
$$
By symmetry, (II) is
$$
\sum_{i\neq j}\mu_i\mu_j a_j b_j
=\mu_e\sum_j\mu_j a_j b_j-\sum_j \mu_j^2 a_j b_j
=\mu_e\sum_i\mu_i a_i b_i-\sum_i \mu_i^2 a_i b_i.
$$
For (III),
\begin{equation*}
\begin{split}
\sum_{i\neq j}\mu_i\mu_j a_i b_j
&=\sum_i \mu_i a_i \sum_{j\neq i}\mu_j b_j
=\sum_i \mu_i a_i\Big(\sum_j \mu_j b_j-\mu_i b_i\Big) \\
&=\Big(\sum_i \mu_i a_i\Big)\Big(\sum_j \mu_j b_j\Big)-\sum_i \mu_i^2 a_i b_i.
\end{split}
\end{equation*}
Again, by symmetry, (IV) is
$$
\sum_{i\neq j}\mu_i\mu_j a_j b_i
=\Big(  \sum_j \mu_j a_j\Big)\Big(\sum_i \mu_i b_i\Big)-\sum_i \mu_i^2 a_i b_i
=\Big(\sum_i \mu_i a_i\Big)\Big(\sum_j \mu_j b_j\Big)-\sum_i \mu_i^2 a_i b_i.
$$

Add (I) and (II), subtract (III) and (IV):

\begin{equation*}
\begin{split}
2S
&=\Big[2\mu_e\sum_i\mu_i a_i b_i-2\sum_i \mu_i^2 a_i b_i\Big]
-\Big[2\Big(\sum_i \mu_i a_i\Big)\Big(\sum_j \mu_j b_j\Big)-2\sum_i \mu_i^2 a_i b_i\Big]\\
&=2\mu_e\sum_i\mu_i a_i b_i
-2\Big(\sum_i \mu_i a_i\Big)\Big(\sum_j \mu_j b_j\Big)
\end{split}
\end{equation*}
and thus
\begin{equation}
    \label{eqn:midStep}
\sum_{i<j}\mu_i\mu_j(a_i-a_j)(b_i-b_j)
=\mu_e\sum_i \mu_i a_i b_i
-\Big(\sum_i \mu_i a_i\Big)\Big(\sum_j \mu_j b_j\Big).
\end{equation}

Recalling that
$$
\sum_i \mu_i a_i=\mu_e a_e,\;\;\mbox{ and }\;\; \sum_j \mu_j b_j=\mu_e b_e,
$$
it follows that
$$
\mu_e\sum_i \mu_i a_i b_i
-\Big(\sum_i \mu_i a_i\Big)\Big(\sum_j \mu_j b_j\Big)
=\mu_e\sum_i \mu_i a_i b_i-\mu_e^2 a_e b_e.
$$
On the other hand,

$$
\begin{aligned}
\mu_e\sum_i \mu_i (a_i-a_e)(b_i-b_e)
&=\mu_e\sum_i \mu_i\big(a_i b_i-a_i b_e-a_e b_i+a_e b_e\big)\\
&=\mu_e\sum_i \mu_i a_i b_i
-\mu_e b_e\sum_i \mu_i a_i
-\mu_e a_e\sum_i \mu_i b_i
+\mu_e a_e b_e\sum_i \mu_i\\
&=\mu_e\sum_i \mu_i a_i b_i
-\mu_e b_e(\mu_e a_e)
-\mu_e a_e(\mu_e b_e)
+\mu_e a_e b_e,\mu_e\\
&=\mu_e\sum_i \mu_i a_i b_i-\mu_e^2 a_e b_e,
\end{aligned}
$$

and thus
$$
\sum_{i<j}\mu_i\mu_j(a_i-a_j)(b_i-b_j)
=\mu_e\sum_i \mu_i (a_i-a_e)(b_i-b_e).
$$

\bibliographystyle{amsalpha}
\bibliography{biblio}      
\end{document}